\documentclass{article}
\usepackage[T1]{fontenc}

\usepackage{amsmath,amssymb,amsfonts,amsthm}
\usepackage{textcomp}
\usepackage{tikz,pgfplots}
\usepackage{pgf,comment,color}

\newcommand{\bm}{\boldsymbol}

\newcommand{\ve}{{\mathrm{vec}}}
\newcommand{\ones}{\bm{1}}
\newcommand{\Gsto}{G_{\mathrm {sto}}}
\newtheorem{theorem}{Theorem}
\newtheorem{lemma}{Lemma}

\newcommand{\tr}{^{\textnormal {\tiny  T}}}
\newcommand{\vect}[1]{\boldsymbol #1}

\newcommand{\vbeta}{\vect \beta}

\newcommand{\vtau}{\vect \tau}
\newcommand{\vone}{\vect 1}

\newcommand{\vs}{\vect s}
\newcommand{\vt}{\vect t}

\newcommand{\1}[1]{\indic[#1]}
\newcommand{\indic}{\mathbb I}

\newcommand{\vligne}[1]{\begin{bmatrix} #1 \end{bmatrix}}

\definecolor{darkmagenta}{rgb}{0.5,0,0.5}
\definecolor{darkgreen}{rgb}{0,0.6,0}
\definecolor{darkblue}{rgb}{0,0,0.6}
\definecolor{darkred}{rgb}{0.8,0,0}
\definecolor{mellow}{rgb}{.847, 0.72, 0.525}

\begin{document}

\title{A family of fast fixed point iterations for M/G/1-type Markov chains\thanks{Research partially supported by INdAM-GNCS}}
\author{Dario Bini\thanks{Dipartimento di Matematica, Universit\`{a}
    di Pisa, Largo Bruno Pontecorvo, 5, 56127 Pisa,  Italy} \and
Guy Latouche\thanks{Universit\'e Libre de Bruxelles,
    D\'epartement d'Informatique, CP212, 1050 Bruxelles, Belgium} \and
Beatrice Meini\thanks{Dipartimento di Matematica, Universit\`{a}
    di Pisa, Largo Bruno Pontecorvo 5, 56127 Pisa,  Italy}}
\maketitle

\begin{abstract}
We consider the problem of computing the minimal nonnegative solution $G$ of 
the nonlinear matrix equation $X=\sum_{i=-1}^\infty A_iX^{i+1}$
where $A_i$, for $i\ge -1$, are nonnegative square matrices such that 
$\sum_{i=-1}^\infty A_i$ is stochastic. This equation is fundamental
in the analysis of M/G/1-type Markov chains, since the matrix $G$
provides probabilistic measures of interest. A new
family of fixed point iterations for the numerical computation of $G$,  that 
includes the classical iterations, is introduced. A detailed convergence analysis  
proves that the iterations in the new class converge faster than the classical 
iterations.  Numerical experiments confirm the effectiveness of our extension.

\end{abstract}



\section{Introduction}
M/G/1-type Markov chains see \cite{neuts81} are characterised by a transition matrix in block Hessenberg form of the kind
\begin{equation}\label{transitionMatrixP}
P=\begin{bmatrix}
B_0&B_1&B_2&\ldots\\
A_{-1}&A_0&A_1&\ldots\\
      &A_{-1}&A_0&\ddots\\
      &&\ddots&\ddots&
\end{bmatrix}
\end{equation}
where $B_i,A_{i-1}$, $i\ge0$ are nonnegative square matrices of order
$m$, such that $\sum_{i=0}^\infty B_i$ and $\sum_{i=-1}^\infty A_i$ are stochastic matrices. 
The matrix $P$ is associated with the nonlinear matrix equation
\begin{equation}\label{eq:mateq}
X=A_{-1}+A_0X+A_1X^2+A_2X^3+\cdots
\end{equation}
where the unknown $X$ is an $m\times m$ matrix.  We may write that
equation as $X = A(X)$, where
$
A(z) = A_{-1}+A_0z+A_1z^2+A_2z^3+\cdots
$.
It is well known that (\ref{eq:mateq}) has a component-wise minimal
nonnegative solution $G$, which, besides having a relevant
probabilistic interpretation, is fundamental in providing an explicit
representation of the invariant probability measure of the Markov
chain, see \cite{blm:book}, and \cite{neuts81}.  Motivated by the important
role of the matrix $G$, numerous algorithms for the numerical
computation of this matrix have been designed and analyzed in the
literature.  A review of the earliest methods for this problem,
specifically in the context of Markov chains, is given in
\cite{rama-rev}, while the more recent analysis in
\cite{higham-kim-2000}, \cite{higham-kim-2001} deals with matrix
quadratic equations in a general framework.  Fixed point iterations
for Markov chains applications are treated in \cite{ren-can-li},
\cite{guo-1999}, \cite{guo-2003}, \cite{rhee-2010}, and Newton's
iteration is the focus in \cite{latouche94}, \cite{benny-2012}, 
\cite{seokim} and \cite{seo2018}.  A vast literature is devoted to
methods based on cyclic reduction and on doubling algorithms; we refer
the reader to the survey paper \cite{bm:cr} and to the literature
cited therein, to the paper \cite{logred} for the logarithmic
reduction algorithm and to \cite{6cinesi} for a convergence analysis
of SDA-based iterations.  A conditioning analysis is performed in
\cite{meng}.

A general overview of the stochastic processes theory behind the
problem, and of the probabilistic aspects of the early algorithms, is
found in the classical books \cite{neuts81} and \cite{lr99}, while the
more numerically oriented book \cite{blm:book} provides a detailed
algorithmic analysis of the problem. In particular, the following
three classical fixed point iterations are analysed in \cite[Chapter
6]{blm:book}; as we shall see, these are special cases of the family introduced and
analysed in the present paper, and they are known as
\begin{align}
&\hbox{Natural} & & X_{k+1}=\sum_{i=0}^\infty A_{i-1}X_k^i, 
\label{eq:nat}\\
&\hbox{Traditional} & &(I-A_0) X_{k+1} =A_{-1}+\sum_{i=2}^\infty A_{i-1}X_k^i,
\label{eq:tra}\\
& \hbox{$U$-based} && \left(I-\sum_{i=0}^\infty A_{i}X_k^i\right)X_{k+1}=A_{-1}, 
\label{eq:uba}
\end{align}
with $k\ge0$, starting from an initial approximation $X_0$.  It is
shown that, if $X_0=0$, the $U$-based iteration converges to $G$
faster than the Traditional iteration which in turn converges faster
than the Natural iteration.
Observe that in all three cases, the matrix function $A(z)$ goes
through a decomposition of the form $A(z)=A_{-1}(z) + A_0(z) z$ where
$A_{-1}(z)$ and $A_0(z)$ depend on the specific iteration, and
(\ref{eq:mateq}) is 
replaced by the equivalent equation 
$
X = A_{-1}(X) + A_0(X)X
$
solved by the fixed point iteration
$
X_{k+1} = A_{-1}(X_k) + A_0(X_k)X_{k+1}
$, for $k=0,1,\ldots$.
For the  Natural iteration we have
$A_{-1}(z) = \sum_{i=0}^\infty A_{i-1}z^i, ~~
 A_0(z) = 0$;
here, the nonlinear part of the equation, the linear factor $A_0$ and the constant
$A_{-1}$ are embedded in the constant term.
For the Traditional iteration we have
$A_{-1}(z)  = A_{-1}+\sum_{i=2}^\infty A_{i-1}z^i, ~~
 A_0(z)  = A_0$,
i.e., the nonlinear part only of (\ref{eq:mateq}) is embedded in the
function $A_{-1}(z)$.
Finally,  for the $U$-based iteration we have

\begin{equation}
  \label{eq:u3}
A_{-1}(z)  = A_{-1},~~
 A_0(z)  = \sum_{i=0}^\infty A_iz^i, 
\end{equation}
so that the nonlinear part is embedded in $A_0(z)$, together with $A_0$.

In summary, the computation of $X_{k+1}$ given $X_k$ is reduced to
solving a {\em linear} matrix equation; the original matrix power
series equation is reduced to a linear equation by embedding the
nonlinear part either in the constant or in the linear coefficient of
equation \eqref{eq:mateq}.

The general strategy suggested by this new interpretation of
well-known algorithms is to reduce the original equation
\eqref{eq:mateq} to a polynomial matrix equation of the kind
\begin{equation}\label{eq:eqqq0}
X=A_{-1}(X)+A_0(X)X+A_1(X)X^2+\cdots+A_q(X)X^{q+1},
\end{equation} 
where $q\ge -1$, and to embed the terms of degree higher than $q+1$ in
the matrix coefficients $A_i(X)$, for $i=-1,\ldots,q$. Indeed, we 
easily verify that \eqref{eq:mateq} is equivalent to
\eqref{eq:eqqq0}
if   the matrix power series
$
A_{\ell}(z)=\sum_{i=0}^\infty A_{\ell,i}z^i,~~\ell=-1,0,1,\ldots,q,
$
satisfy the only conditions $A_{\ell,i}\ge 0$ and
$
\sum_{\ell=-1}^q A_{\ell}(z)z^{\ell+1}=\sum_{i=0}^\infty A_{i-1}z^i.
$

Equation \eqref{eq:eqqq0} is solved by the fixed point iteration
\begin{equation}\label{eq:fixquad20}
X_{k+1}=\sum_{\ell=-1}^q A_{\ell}(X_k)X_{k+1}^{\ell+1},~~k=0,1,\ldots,
\end{equation}
for a given initial approximation $X_0$, where the matrix $X_{k+1}$ is defined as the minimal
nonnegative solution of the matrix equation \eqref{eq:fixquad20}.

As already pointed out, 
the classical fixed point
iterations  
can be viewed  as specific instances of the new class of iterations
\eqref{eq:fixquad20}: the Natural iteration is  obtained by choosing
$q=-1$ and embedding all the terms in $A_{-1}(z)$, $q=0$  for the Traditional
iteration and all the terms of degree higher than 2 are
embedded  in $A_{-1}(z)$ while for the $U$-based iteration, $q=0$ also but all
the terms of degree higher than 2 are embedded in $A_0(z)$.

Observe that, if $q \geq 1$, then the polynomial equation
\eqref{eq:fixquad2} has degree at least 2 and its solution
requires an iterative technique. Therefore, the algorithms in the new
class may be viewed as formed by an outer iteration, which generates
$X_{k+1}$ given $X_k$, and an inner iteration needed to approximate
$X_{k+1}$ by solving the polynomial equation of degree $q+1$.

We present in this paper our analysis of this new class of fixed point
iterations. We rely on the properties of
nonnegative matrices and of regular splittings of M-matrices, and show
that the sequence $\{X_k\}_k$ given by \eqref{eq:fixquad2} is well
defined and monotonically convergent to the solution $G$ if $X_0=0$,
and convergent if  $X_0$ is any stochastic matrix.  Moreover, we show
that  for a
given degree $q+1$, the highest speed of convergence is
obtained by choosing 
$
A_\ell(z)  =A_\ell, ~ \ell=-1,0,\ldots,q-1,\quad
A_q(z) =\sum_{i=0}^\infty A_{i+q}z^{i},
$
that is, by embedding all $A_i$s, $i \geq q$, in $A_q(z)$.
For this particular family of equations, we show that the convergence
speed is higher for larger values of $q$, and so,
for $q\ge 1$, all
the iterations in this class outperform the traditional iterations
(\ref{eq:nat} -- \ref{eq:uba}).

A further analysis shows that the polynomial matrix equation of degree
$q+1$ is better conditioned and easier to solve for smaller values of  $q$.
Thus, the outer iteration is faster the larger $q$ is,
whereas the inner iteration is faster the smaller $q$ is. Apparently,
this provides a trade-off in the choice of $q$.  However, as we will
see from the numerical experiments, the number of inner iterations can
be substantially reduced by a suitable choice of the starting
approximation, and it turns out that in practice, the number of inner
iterations is a nonincreasing function of $q$. This fact makes the new
class of iterations not only faster in terms of convergence speed, but
also more effective in terms of CPU time. In fact, our numerical
experiments show that the speed-up in the CPU time, with respect to
the available iterations, is generally larger than 2 and in many cases
reaches values larger than 20.

The paper is organized as follows. In Section~\ref{sec:prel} we recall
the main tools used in our analysis. In Section~\ref{sec:convq} we
analyse the convergence of the new family of iterations in the case
where $X_0=0$. Different embedding strategies are discussed in
Section~\ref{sec:opt}, with the aim to determine an optimal value of
$q$.  In Section~\ref{sec:stoc}, we show convergence properties in the
case where $X_0$ is a stochastic matrix.  A computational cost and
conditioning analysis is performed in Section~\ref{sec:cost}, and we
report in Section~\ref{sec:num} on our numerical experiments.

\section{Preliminary results}\label{sec:prel}

In this section we recall results on nonnegative matrices and M/G/1-type Markov chains that will be used in the rest of the paper.

\subsection{Nonnegative matrices}
A real matrix $A$ is called nonnegative if $A\ge 0$, where the inequality is meant entry-wise. If $A$ and $B$ are real nonnegative matrices of the same size, the inequality $A\ge B$ means $A-B\ge 0$. Given a complex matrix $B$, we denote by $|B|$ the matrix whose entries are the  moduli of the entries of $B$, while $\rho(B)$ denotes the
 spectral radius of $B$, i.e., the maximum modulus of its eigenvalues.
 
Given a matrix power series $A(z)=\sum_{i=0}^\infty A_iz^i$ with real matrix coefficients, we write $A(z)\ge 0$ if $A_i\ge 0$ for any $i\ge0$; moreover, if $B(z)$ is a matrix power series, we write $A(z)\ge B(z)$ if $A(z)-B(z)\ge 0$.   

A matrix of the kind $B=sI-A$, where $A\ge 0$ and $s\ge \rho(A)$, is
called an M-matrix. It is nonsingular if $s> \rho(A)$, otherwise it is a singular M-matrix.
An additive splitting $A=M-N$ of the real square matrix $A$ is called a regular splitting if $\det M\ne 0$, $M^{-1}\ge 0$ and $N\ge 0$.

The following result
synthesizes some properties from the Perron-Frobenius theory of nonnegative matrices \cite{bp:book}.

\begin{theorem} \label{thm:pf}
Let $A\ge 0$ be a square matrix. The following properties hold:
\begin{itemize}
\item $\rho(A)$ is an eigenvalue; if, in addition, $A$ is irreducible, then $\rho(A)>0$;
\item if $B\ge A$ then $\rho(B)\ge \rho(A)$; if $A$ is irreducible and $B\neq  A$, then $\rho(B)>\rho(A)$;
\item if $B$ is a complex matrix such that $|B|\le A$, then $\rho(B)\le \rho(A)$.
\end{itemize}
\end{theorem}

We provide some properties of M-matrices and their regular splittings \cite{varga}.

\begin{theorem}\label{thm:rs0}
Assume that $B$ is a nonsingular  M-matrix. Then:
\begin{itemize}
\item $B^{-1}\ge 0$;
\item if $B=M-N$ is a regular splitting, then 
$
\rho(M^{-1}N)=\rho(B^{-1}N)/(1+\rho(B^{-1}N))<1;
$
\item if $B=M_1-N_1=M_2-N_2$ are two regular splittings and $N_1\le N_2$, then $\rho(M_1^{-1}N_1)\le \rho(M_2^{-1}N_2)$; if $M_1^{-1}N_1\ne M_2^{-1}N_2$
and
$M_1^{-1}N_1$ is irreducible
then
the inequality is strict.
\end{itemize}
\end{theorem}

The following result on matrix polynomials is a consequence of Theorem~\ref{thm:pf}.

\begin{lemma}\label{lem:comp}
Let $B(z)=z^nI - \sum_{i=0}^{n-1} B_i z^i$ and
$\tilde B(z)=z^nI - \sum_{i=0}^n \tilde B_i z^i$
be  monic matrix polynomials such that 
$0\le\tilde B_i\le B_i$
 for $i=0,\ldots,n-1$.  Then the polynomials $\det B(z)$ and  $\det \tilde B(z)$ 
  have a nonnegative real root $\lambda$ and $\tilde\lambda$,
  respectively, which is the root of largest modulus. Moreover,  $\tilde \lambda\le \lambda$.
\end{lemma}

\begin{proof}
The roots of $\det B(z)$ and of $\det \tilde B(z)$ are the eigenvalues of the block companion matrices $C$ and $\tilde C$ associated with the monic matrix polynomials $B(z)$ and $\tilde B(z)$, respectively \cite{gohberg2009matrix}. Since $0\le\tilde C\le C$,  from Theorem~\ref{thm:pf} we deduce that $\lambda=\rho(C)$ and $\tilde \lambda=\rho(\tilde C)$, moreover $\rho(\tilde C)\le \rho(C)$.

\end{proof}

\subsection{M/G/1-type Markov chains}
Assume the M/G/1-type Markov chain with transition matrix \eqref{transitionMatrixP}
is irreducible and aperiodic, and that $A=\sum_{i=-1}^{\infty}A_i$ is irreducible and stochastic.
Assume also that the series $\sum_{i=-1}^\infty iA_i$ is convergent
and define the vector $\bm a=\sum_{i=-1}^\infty iA_i\ones$, where
$\ones$ is the vector with all entries equal to 1. Let $\bm\alpha$ be
the vector such that $\bm\alpha\tr A=\bm\alpha\tr$,
$\bm\alpha\tr\ones=1$; the drift of the Markov chain is defined as
\cite{blm:book,neuts81}
\begin{equation}\label{eq:mu}
\mu=\bm\alpha\tr\bm a
\end{equation}
and  $\mu\le 0$ if and only if the Markov chain is recurrent, $\mu>0$ if and only if it is transient.

The minimal nonnegative solution of the matrix equation
\eqref{eq:mateq} is characterised  as follows \cite{blm:book}.

\begin{theorem}\label{thm:exist}
Let $A_i$, $i=-1,0,1,\ldots$, be nonnegative square matrices such that $(\sum_{i=-1}^\infty A_i)\ones\le\ones$. The matrix equation \eqref{eq:mateq} has a unique  minimal nonnegative solution $G$, i.e., if $Y$ is any other nonnegative solution, then $G\le Y$. Moreover, $G\ones\le\ones$ and $G$ is the limit of the sequence 
$X_{k+1}=A_{-1}+A_0X_k+A_1X_k^2+\cdots$, $k=0,1,\ldots$,
with $X_0=0$.
\end{theorem}

The next result provides a comparison between the minimal nonnegative solutions of two matrix equations.

\begin{theorem}\label{thm:comparison}
Let $A_i$ and $\widetilde A_i$, $i=-1,0,1,\ldots$, be nonnegative square matrices such that $\widetilde A_i\le A_i$ for any $i\ge -1$ and $(\sum_{i=-1}^\infty A_i)\ones\le\ones$. 
Let $G$ and $H$ be the minimal nonnegative solutions of the matrix equations \eqref{eq:mateq} and $X=\sum_{i=0}^\infty \widetilde A_{i-1}X^i$, respectively. Then $H\le G$.
\end{theorem}

\begin{proof}
We have $G=\lim_{k\to\infty}X_k$, where $X_k$ is defined in Theorem \ref{thm:exist}. 
Similarly, $H=\lim_{k\to\infty}Y_k$, where $Y_k$ is defined by  $Y_{k+1}=\sum_{i=0}^\infty \widetilde A_{i-1}Y_k^i$, with $Y_0=0$. 
By induction on $k$, one has $Y_k\le X_k$, for $k\ge 0$, therefore the inequality holds also in the limit.
\end{proof}

For the convergence analysis of fixed point iterations, it is useful to introduce the matrices
\begin{equation}\label{eq:H}
V=\sum_{j=0}^\infty A_j^*,\qquad A_i^*=\sum_{j=i}^\infty A_j G^{j-i},~~i\ge 0.
\end{equation}
If $\mu<0$  then $\rho(V)<1$, so
that $H=I-V$ is a nonsingular M-matrix \cite{blm:book}.

The following result see \cite{blm:book}, concerns the convergence of fixed point iterations when the starting approximation $X_0$ is stochastic. 

\begin{theorem}\label{thm:existsto}
  Let $A_i$, $i=-1,0,1,\ldots$, be nonnegative square matrices such
  that $(\sum_{i=-1}^\infty A_i)\ones=\ones$. The matrix equation
  \eqref{eq:mateq} has a unique stochastic solution $\Gsto$. It is the
  limit of the sequence $\{X_k\}$ defined in Thorem \ref{thm:exist} 
  with $X_0$ stochastic, and
  $G\le \Gsto$, where $G$ is the minimal nonnegative solution of
  \eqref{eq:mateq}. Finally, if $\mu\le 0$ then $\Gsto=G$.
\end{theorem}

\section{A new family of fixed point iterations}\label{sec:convq} 
In this section we introduce the new family of fixed point iterations, we prove the well-posedness, analyze their convergence properties  and, for a given $q$,  determine the optimal choice of the matrix power series $A_\ell(z)$ that maximizes the convergence rate. Moreover, we give a physical interpretation of the new family in terms of Markov chain properties.

Let $q\ge -1$ and let 
\begin{equation}\label{eq:coeff2}
A_{\ell}(z)=\sum_{i=0}^\infty A_{\ell,i}z^i,~~\ell=-1,0,1,\ldots,q,
\end{equation}
be  matrix power series such that
 $A_{\ell,i}\ge 0$ and
\begin{equation}\label{eq:sum2}
\sum_{\ell=-1}^q A_{\ell}(z)z^{\ell+1}=\sum_{i=-1}^\infty A_{i}z^{i+1}.
\end{equation} 
Equating the coefficients of $z^i$ on both sides of \eqref{eq:sum2}
yields
\begin{equation}\label{eq:relc2}
A_i=A_{-1,i+1}+A_{0,i}+A_{1,i-1}+\cdots+A_{q,i-q}, \quad i\ge -1, 
\end{equation} 
where, for $\ell=-1,0,\ldots,q$, we set $A_{\ell,i}=0$ if $i< 0$. This means that the probability to jump $i$ levels, represented by the matrix $A_i$, is spread into the probability to jump $\ell$ levels, where $\ell$ ranges from $-1$ to $q$.

By replacing the variable $z$ in \eqref{eq:sum2} with the matrix $X$,  we easily check that 
 the original equation
\eqref{eq:mateq} is equivalent to the matrix equation 
$X=A_{-1}(X)+A_0(X)X+A_1(X)X^2+\cdots+A_q(X)X^{q+1}$.
If we interpret $A_i(X)$, for $i=-1,\ldots,q$, as matrix coefficients, then the latter matrix equation can be seen as a matrix polynomial equation of degree $q+1$.
The new family of fixed point iterations consists in solving the above equation  
by means of the fixed point iteration
\begin{equation}\label{eq:fixquad2}
X_{k+1}=\sum_{\ell=-1}^q A_{\ell}(X_k)X_{k+1}^{\ell+1},~~k=0,1,\ldots,
\end{equation}
for a given initial approximation $X_0$, where the matrix $X_{k+1}$ is defined as the minimal
nonnegative solution of the matrix equation \eqref{eq:fixquad2}.
We prove with the next theorem that, if $X_0=0$, the sequence $\{X_k\}_k$ generated by \eqref{eq:fixquad2}  is well defined and converges monotonically to the solution $G$ of \eqref{eq:mateq}.

\begin{theorem}\label{thm:conv2}
Set $X_0=0$ and, for $k\ge 0$, define $X_{k+1}$ as the minimal
nonnegative solution of \eqref{eq:fixquad2}. 
Then the sequence $\{X_k\}_k$ is well defined,  $0\le X_{k}\le X_{k+1}$ and $X_{k+1}\ones\le\ones$ for $k=0,1,\ldots$. 
Moreover, the sequence $\{X_k\}_k$ converges monotonically to the minimal nonnegative solution $G$ of \eqref{eq:mateq}.
\end{theorem}

\begin{proof}
We first observe that, since $A_{\ell,i}\ge 0$ for any $\ell$ and $i$,
then $0\le A_\ell(X)\le A_\ell(Y)$ whenever $0\le X\le Y$. We prove
the theorem by induction on $k$. If $k=0$ then $A_\ell(0)\ge0$, for
$\ell=-1,0,\ldots,q$, and $\sum_{\ell=-1}^q A_{\ell}(0)\ones\le\ones$,
therefore the minimal nonnegative solution $X_1$ exists  and
$X_1\ones\le \ones$ by  Theorem \ref{thm:exist}.
Assume that $0\le X_{k}\le X_{k+1}$ and $X_{k+1}\ones\le\ones$. Then  $0\le A_\ell(X_k)\le A_\ell(X_{k+1})$. Moreover, since $X_{k+1}\ones\le\ones$, then $\sum_{\ell=-1}^q A_{\ell}(X_{k+1})\ones\le\ones$. Therefore, according to Theorems \ref{thm:exist} and \ref{thm:comparison}, the minimal nonnegative solution $X_{k+2}$ exists, $0\le X_{k+1}\le X_{k+2}$ and $X_{k+2}\ones\le\ones$. Since the sequence $\{X_k\}_k$ is monotonic non-decreasing and bounded from above, it is convergent to a limit $G$, which solves equation \eqref{eq:mateq} by continuity. Such limit is the minimal nonnegative solution since, if $Y\ge 0$ is any other nonnegative solution, then we easily prove by induction on $k$ that $X_k\le Y$ for any $k$. Therefore, taking the limit yields $G\le Y$. 
\end{proof}
 
To further analyse the  convergence, we need the following technical
lemma, which can be easily proved by induction.

\begin{lemma}
   \label{t:tech} 
If $X$ and $Y$ are square matrices of the same size, and $n\ge 0$ is an integer, then
$X^n-Y^n=\sum_{j=0}^{n-1}X^j(X-Y)Y^{n-j-1}$.
\end{lemma}

The following result provides a relation between the error $E_k=G-X_k$ at two subsequent steps.
 
\begin{theorem}\label{th:11}
Let $E_k=G-X_k$, $k\ge 0$, where the sequence $\{ X_k\}$ is defined by
means of \eqref{eq:fixquad2} for any given $X_0$. One has
\begin{equation}\label{eq:err1}
\begin{split}
&E_{k+1}=\sum_{\ell=0}^qA_\ell(G)\sum_{j=0}^\ell G^j E_{k+1}X_{k+1}^{\ell-j} +S(E_k),\\
&S(E_k)=\sum_{\ell=-1}^q
\sum_{i=1}^\infty A_{\ell,i} \sum_{j=0}^{i-1}G^j E_k X_k^{i-j-1}X_{k+1}^{\ell+1}.
\end{split}\end{equation}
\end{theorem}

\begin{proof}
By subtracting \eqref{eq:fixquad2} from the equation $G=\sum_{\ell=-1}^q A_{\ell}(G)G^{\ell+1}$, we obtain
\[
E_{k+1}= 
\sum_{\ell=-1}^q 
\left(A_{\ell}(G)\left(G^{\ell+1}-X_{k+1}^{\ell+1}\right)+
(A_\ell(G)-A_\ell(X_k))X_{k+1}^{\ell+1}\right)
.\]
From \eqref{eq:coeff2} and Lemma \ref{t:tech}, we find that 
$G^{\ell+1}-X_{k+1}^{\ell+1}=\sum_{j=0}^\ell G^j E_{k+1}X_{k+1}^{\ell-j}$, and
$ A_\ell(G)-A_\ell(X_k)=\sum_{i=1}^\infty A_{\ell,i} \sum_{j=0}^{i-1}G^j E_k X_k^{i-j-1}$,
$\ell=-1,0,1,\ldots,q$,
hence we arrive at \eqref{eq:err1}.
\end{proof}

If $X_0=0$ then, in view of Theorem \ref{thm:conv2}, $E_k\ge 0$ so that $\|E_k\|_\infty=\|E_k\ones\|_\infty$. Therefore, 
we may estimate the convergence rate of the sequence $\{X_k\}_k$ by analyzing  the convergence of the sequence $\bm\epsilon_k=E_k\ones \ge 0$. The following theorem provides information in this regard.

\begin{theorem}\label{th:12}
 If $X_0=0$, then
 $M\bm\epsilon_{k+1}\le N \bm\epsilon_k$, where  
\begin{equation}\label{eq:MN2}
M=I-F,
\quad
 N= \sum_{\ell=-1}^q\sum_{i=1}^\infty A_{\ell,i}\sum_{j=0}^{i-1}G^j,
\quad 
F=\sum_{\ell=0}^q A_\ell(G)\sum_{j=0}^\ell G^j.
\end{equation}
Moreover, $M-N=H$, where $H=I-V$ and $V$ is defined in \eqref{eq:H}. If the drift $\mu$ of \eqref{eq:mu} is negative, then $\rho(F)<1$,
$M$ is a nonsingular M-matrix and $H=M-N$ is a regular splitting, therefore
$\rho(M^{-1}N)=\rho(H^{-1}N)/(1+\rho(H^{-1}N))<1$.
\end{theorem}

\begin{proof}
Since $X_0=0$,  it results from  Theorem \ref{thm:conv2} that $0\le
X_k\le G$ and $X_k\ones\le \ones$ and, multiplying \eqref{eq:err1} on the right by $\ones$, we obtain
\[
\bm\epsilon_{k+1}\le 
\sum_{\ell=0}^qA_\ell(G)\sum_{j=0}^\ell G^j
\bm \epsilon_{k+1} +S(E_k)\ones.
\]
On the other hand $S(E_k)\ones\le
\sum_{\ell=-1}^q
\sum_{i=1}^\infty A_{\ell,i} \sum_{j=0}^{i-1}G^j
\bm\epsilon_k$ whence
$
M\bm\epsilon_{k+1}\le N \bm\epsilon_k
$.
By using equation \eqref{eq:relc2},
we find that 
\[
\begin{split}
I-(M-N)=&\sum_{\ell=-1}^q A_\ell(G)\sum_{j=0}^\ell G^j+ 
\sum_{\ell=-1}^q\sum_{i=1}^\infty A_{\ell,i}\sum_{j=0}^{i-1}G^j\\
 =& \sum_{\ell=-1}^q\left( \sum_{i=0}^\infty A_{\ell,i}G^i
\sum_{j=0}^\ell G^j + 
 \sum_{i=1}^\infty A_{\ell,i}
\sum_{j=0}^{i-1} G^j\right),
\end{split}
\]
where in the first equality we used the fact that $\sum_{j=0}^\ell
G^j=0$ if $\ell=-1$, 
whence we get
$I-(M-N) =\sum_{\ell=-1}^q \sum_{i=0}^\infty A_{\ell,i} \sum_{j=0}^{i+\ell}G^j=
\sum_{\ell=-1}^q \sum_{k=\ell}^\infty A_{\ell,k-\ell} \sum_{j=0}^{k}G^j
$.
By using the convention that $A_{\ell,i}=0$ if $i<0$ and the property
that all the series are absolutely convergent, we may exchange the
order of the summations in the last term of the above equation, and,
by using  \eqref{eq:relc2} again, arrive at
$
I-(M-N)=\sum_{k=-1}^\infty \sum_{\ell=-1}^q  A_{\ell,k-\ell}
\sum_{j=0}^k G^j=\sum_{k=-1}^\infty A_k\sum_{j=0}^k G^j=V.
$ 
To show that $H=M-N$ is a regular splitting, we need to prove that
$\det M\ne 0$ and $M^{-1}\ge 0$. Since $M=I-F$, where $0\le F\le V$
and $\rho(V)<1$ if $\mu<0$ see \cite[Theorem 4.14]{blm:book}, we have
$\rho(F)\le \rho(V)<1$ by Theorem~\ref{thm:pf}, so that $M$ is a
nonsingular M-matrix and $M^{-1}\ge 0$.  As $N$ is nonnegative,
$H=M-N$ is a regular splitting, and we obtain the
expression for $\rho(M^{-1}N)$ from 
Theorem~\ref{thm:rs0}.
\end{proof}

The spectral radius of $M^{-1}N$ provides an estimate of the convergence rate of the sequence $\{X_k\}_k$, when $X_0=0$. Indeed, from Theorem \ref{th:12}, 
$
0\le\bm\epsilon_k\le (M^{-1}N)^k \bm\epsilon_0,
$
 therefore 
$
\|\bm\epsilon_k\|_\infty \le \| (M^{-1}N)^k\|_\infty  \| \bm\epsilon_0\|_\infty$.
The ratio 
$\|\bm\epsilon_k\|_\infty/\|\bm\epsilon_{k-1}\|_\infty$ represents the norm reduction of the error at step $k$, while the geometric mean of the reduction of the errors in the first $k$ steps, i.e.,
\[
r_k:=\left(\frac{\|\bm\epsilon_1\|}{\|\bm\epsilon_0\|}\frac{\|\bm\epsilon_2\|}{\|\bm\epsilon_1\|}\cdots \frac{\|\bm\epsilon_k\|}{\|\bm\epsilon_{k-1}\|}\right)^\frac1k
=\left(\frac{\|\bm\epsilon_k\|}{\|\bm\epsilon_0\|}\right)^\frac1k,
\]
represents the average reduction of the errors per step after $k$ steps. Observe that,  $r_k\le \| (M^{-1}N)^k\|_\infty^{1/k}$. 
By following \cite{guo-1999},  \cite{meini:fi}, we define the asymptotic rate of convergence $r=\limsup_k r_k$. From the latter inequality and  the property $\lim_k\| (M^{-1}N)^k\|_\infty^{1/k}=\rho(M^{-1}N)$, we find that
$r \le\rho(M^{-1}N)$.

In consequence of Theorem~\ref{th:12}, we next compare the speed of convergence of different iterations by comparing the corresponding matrix $F$ of \eqref{eq:MN2}.

\begin{theorem}\label{thm:conf}
Let $\{X_k^{(h)}\}_k$, for $h=1,2$, be two sequences generated by \eqref{eq:fixquad2}, with $X_0^{(1)}=X_0^{(2)}=0$, defined by 
$A_\ell^{(h)}(z)=\sum_{i=0}^\infty A_{\ell,i}^{(h)}z^i$, $\ell=-1,\ldots,q_h$, for $h=1,2$. Let $r^{(h)}$, $h=1,2$, be their asymptotic rates of convergence.
If 
\begin{equation}\label{eq:comp}
\sum_{\ell=0}^{q_1} A_\ell^{(1)}(G)\sum_{j=0}^\ell G^j\ge \sum_{\ell=0}^{q_2} A_\ell^{(2)}(G)\sum_{j=0}^\ell G^j,
\end{equation}
then $r^{(1)}\le r^{(2)}$, i.e., the sequence $\{X_k^{(1)}\}_k$ converges faster than the sequence $\{X_k^{(2)}\}_k$.
\end{theorem}

\begin{proof}
By Theorem \ref{th:12}, $H=M^{(h)}-N^{(h)}$, $h=1,2$, are two regular splittings. 
Therefore, if $N^{(1)}\le N^{(2)}$, then
 $\rho({M^{(1)}}^{-1}N^{(1)})\le \rho({M^{(2)}}^{-1}N^{(2)})$.
On the other hand, $N^{(1)}\le N^{(2)}$ is equivalent to $I-M^{(1)}\ge I-M^{(2)}$, i.e., equivalent to \eqref{eq:comp}.
\end{proof}

As a corollary of Theorem~\ref{thm:conf}, we obtain the results shown
in \cite{meini:fi}, whereby the $U$-based iteration \eqref{eq:uba} is
faster than the traditional iteration \eqref{eq:tra}, which is in turn
faster than the natural iteration \eqref{eq:nat}. Indeed, for these
three iterations $q\le 0$ and, denoting by $F^{(N)}$, $F^{(T)}$ and
$F^{(U)}$ the matrix $F$ in (\ref{eq:MN2}) for the three iterations,
we find that
$
F^{(N)} = 0 \le F^{(T)} = A_0 \le F^{(U)} = \sum_{i=0}^\infty A_i G^i$.

For the next theorem, we assume that $A_{-1}(z)=A_{-1}$: this means
that the transition probability matrices $A_0$, $A_1$, \ldots may be
variously embedded in the coefficients of order 0 to $q$, but none in
$A_{-1}(G)$.  We show below that any iteration of the kind
\eqref{eq:fixquad2}, with $q\ge 1$, which satisfies this constraint
converges faster than the $U$-based iteration; we discuss at the end
of Section \ref{sec:fis} the physical significance of this assumption.

If $A_{-1}(z)=A_{-1}$, Equation \eqref{eq:sum2} may be rewritten as
\begin{equation}\label{eq:alz}
\sum_{\ell=0}^qA_\ell(z)z^\ell=\sum_{i=0}^\infty A_iz^i,
\end{equation}
so that, by replacing $z$ with $G$,  
we deduce that
\begin{equation}\label{eq:un}
\sum_{\ell=0}^qA_\ell(G)G^\ell=\sum_{i=0}^\infty A_iG^i.
\end{equation}

\begin{theorem}\label{th:9}
Let $q\ge 1$, assume that $X_0=0$,  and that $A_{-1}(z)=A_{-1}$.  The sequence $\{ X_k\}_k$ generated by \eqref{eq:fixquad2} converges faster than the sequence \eqref{eq:uba} generated by the $U$-based iteration.
\end{theorem}

\begin{proof}
The proof is a consequence of Theorem~\ref{thm:conf}, where $\{X_k^{(1)}\}_k$ is the sequence defined by \eqref{eq:fixquad2}, and $\{X_k^{(2)}\}_k$ is the sequence \eqref{eq:uba}. 
From \eqref{eq:u3}, we have $A^{(2)}_{0}(G)=\sum_{i=0}^\infty A_iG^i$
and $A^{(2)}_{\ell}(G)=0$ for $\ell\ge 1$, therefore the inequality
\eqref{eq:comp} to be verified is equivalent to
$
\sum_{\ell=0}^q A_\ell(G)\sum_{j=0}^\ell G^j\ge \sum_{i=0}^\infty A_iG^i,
$
this clearly holds by \eqref{eq:un}.
\end{proof}

\subsection{Physical interpretation}\label{sec:fis}
We may give an interpretation of Theorems
  \ref{thm:conf} and \ref{th:9}, based on the physical significance of the matrices $V=(v_{ij})_{ij}$ and $F=(f_{i,j})_{ij}$.
 Assume that the Markov chain with transition matrix
  \eqref{transitionMatrixP} starts at time 0 in an arbitrary but fixed level
  $n \geq 1$.  Define $\tau_{-1}$, $\tau_0$, $\tau_1$, $\ldots$ to be the
  epochs when the Markov chain makes a first transition to level
  $n+i$, for $i \geq -1$, and define 
\[
N_{j} = \sum_{\nu \geq 0}  \1{\tau_\nu < \tau_{-1}, \varphi(\tau_\nu) = j},
\]
for $j = 1, \ldots, m$, where $\1{\cdot}$ is the indicator function and $\varphi(t)$ represents the phase occupied by the Markov chain at time $t$; that is, $N_j$ is the total number of
times the Markov chain visits phase $j$, at the epochs of first
visit to a new level, under taboo of the levels below level $n$.  
It is easy to verify that $v_{ij}$ is the conditional expected number
of such visits, given that the initial phase is $i$.  Similarly,
$f_{ij}$ is the expected number of such visits for the Markov chain
where the transition blocks $A_k$ are replaced by $A_k(G)$, for
$k \geq -1$, and we consider it as providing an approximation of $v_{ij}$.

From Theorem \ref{th:12}, we know that $F \leq V$, the
inequality \eqref{eq:comp} may be rewritten as $V \geq F^{(1)} \geq
F^{(2)}$, and so Theorem  \ref{thm:conf} states that $\{X_k^{(1)}\}_k$
converges faster than $\{X_k^{(2)}\}_k$ if the approximation
$f_{ij}^{(1)}$ is uniformly better than $f_{ij}^{(2)}$ for all
$i$ and $j$.

Define 
\begin{equation}
   \label{e:nprime}
N'_j = \1{\tau_0 < \tau_{-1}, \varphi(\tau_0) = j} \leq N_j.
\end{equation}
The expected value of $N'_j$ is the probability of returning to the
initial level $n$ in phase $j$, under taboo of level $n-1$, and
\eqref{eq:un}  shows that this probability is the same for all
embedding $A_{\ell,k}$, $\ell \geq 0$, if $A_{-1}(G) = A_{-1}$, that
is, provided that the probability of transiting immediately to the
lower level remains unchanged.

Now, we have for the $U$-based
iteration
$A_{-1}^{(U)}  = A_{-1}$, $A_{0}^{(U)} = \sum_{k \geq 0} A_k G^k$,
$A_\ell^{(U)} = 0$, for $\ell \geq 2$,
by \eqref{eq:u3}, and we readily verify that $F^{(U)}$ may also be
interpreted as the matrix of expected values of $N'$ and
(\ref{e:nprime}), together with Theorem~\ref{thm:conf} provides a
physical justification for Theorem~\ref{th:9}.

\section{Optimal embedding}\label{sec:opt}
In this section we examine the role of the integer $q$ which
determines the degree of the matrix equation \eqref{eq:fixquad2} to be
solved at each step. In particular, the goal is to give properties to
determine an optimal value of $q$, in terms of speed of convergence of
the sequence $\{ X_k\}_k$ and in terms of numerical properties of the
matrix equation \eqref{eq:mateqq}.

\subsection{Comparisons}\label{sec:comp}
To embed the tail of the series only in the coefficients of the terms
of degrees $0$ to $q+1$, for a given integer $q \geq 1$,  according to \eqref{eq:relc2},
the
matrices $A_{\ell,k}$ satisfy

\begin{align}
   \label{e:un}
A_\nu & = A_{-1,\nu+1} + A_{0,\nu} + \cdots +  A_{\nu,0}, &&\mbox{for $-1 \leq \nu \leq q$},
\\
   \label{e:deux}
A_\nu & =  A_{-1,\nu+1} + A_{0,\nu} + \cdots +  A_{q,\nu-q} , &&\mbox{for $\nu \geq q+1$,}
\\
A_{\ell, i} & = 0, && \mbox{for $\ell \geq q+1$, $i \geq 0$.}
\end{align}

\begin{theorem}\label{thm:4.1}
Given an integer $q \geq 1$, the parameters that maximize the matrix $F$ of \eqref{eq:MN2}
are given by

\begin{align}
   \label{e:unbis}
A_{\ell, 0} & = A_\ell, &&\mbox{for~ $-1 \leq \ell \leq q$,}
\\
   \label{e:deuxbis}
A_{q, \ell-q} & = A_\ell, &&\mbox{for~ $ \ell \geq q+1$,}
\\
   \label{e:troisbis}
A_{\ell, i} & = 0, &&\mbox{for~ $-1 \leq \ell \leq q-1$, $i \geq 1$. }
\end{align}
In other words, 

\begin{equation}
\label{eq:mass}
A_\ell(z)=A_\ell,   \quad \mbox{for~ $-1 \leq \ell \leq q-1$,}\qquad 
A_q(z)=\sum_{i=q}^\infty A_i z^{i-q}.
\end{equation}
and the tail of the series is embedded in the coefficient
of degree $q+1$ only.
\end{theorem}
\begin{proof}
The sum $F$ of \eqref{eq:MN2} may be re-written as
\[
F  = \sum_{0 \leq k \leq q}  \ \sum_{i \geq 0} A_{k,i}  \sum_{i \leq
       \nu \leq k+i} G^\nu= \sum_{\ell \geq 0}  \ \sum_{0 \leq k \leq \min(\ell,q)}
   A_{k,\ell-k}  \sum_{\ell - k \leq \nu \leq \ell} G^\nu
   \]
which coincides with   
  \[
   \sum_{\ell \geq 0}  \ \sum_{0 \leq k \leq \min(\ell,q)}
   A_{k,\ell-k}  B_{k,\ell},
  \]
where $B_{k, \ell}= \sum_{\ell - k \leq \nu \leq \ell} G^\nu$ is
increasing with $k$, for any given $\ell$.  Thus, it suffices
to use the matrices in (\ref{e:unbis} -- \ref{e:troisbis}) to maximise
$F$ under the constraints (\ref{e:un}, \ref{e:deux}).
Finally, if we equate the coefficients of $z^0$ in (\ref{eq:sum2}), we
find that $A_{-1,0} = A_{-1}$, and this  completes the proof.
\end{proof}

\subsection{Spectral properties}\label{sec:sp}
In this analysis we restrict the attention to the case where $A_\ell(z)\ge A_\ell$, for $\ell=0,1,\ldots,q$, where we recall that the inequality involving matrix power series is meant coefficient-wise.
We show some spectral properties which clarify the role of  $q$ in the numerical properties  of the matrix equation
\begin{equation}\label{eq:mateqq}
X=A_{-1}(G)+A_0(G) X+A_1(G) X^2+\cdots+A_{q-1}(G)X^q+A_q(G)X^{q+1},
\end{equation}
in terms of conditioning and speed of convergence of 
fixed point iterations.

Define the matrix Laurent power series 
$S(z)=I-\sum_{i=-1}^\infty A_iz^i$, and the polynomial
$S_q(z)=I-\sum_{i=-1}^{q} A_i(G)z^i,
$
associated with the matrix equations
 \eqref{eq:mateq} and 
\eqref{eq:mateqq}
respectively.
If  the drift $\mu$ is negative, it is well known see \cite[Theorem 4.12]{blm:book} that there exists $\xi>1$ such that $\det S(\xi)=0$ and
$
\xi=\min\{|z|: z\in\mathbb{C},|z|>1,\det S(z)=0\}$.
The closeness of $\xi$ to 1 governs the convergence of numerical
methods for solving the matrix equation \eqref{eq:mateq}, as well as
the conditioning of the problem: the closer is $\xi$ to 1, the slower
is the convergence of numerical methods and the worse is the
conditioning \cite[Chapter 7]{blm:book}.

We show in the theorem below that the smallest root of $\det S_q(z)$
outside the closed unit disk is larger than $\xi$, and so the matrix
equation \eqref{eq:mateqq} has better numerical properties than the
original equation \eqref{eq:mateq}, if $A_\ell(z)\ge A_\ell$,
$0\leq \ell \leq q$, and $A_{-1}(z) = A_{-1}$.  In view of Theorem
\ref{thm:4.1}, this is not a very restrictive assumption.

\begin{theorem}\label{thm:fact}
Assume  $q\ge 1$.  If $A_\ell(z)\ge A_\ell$, for
$\ell=0,1,\ldots,q$, and $A_{-1}(z)=A_{-1}$, then the matrix functions $S(z)$ and $S_q(z)$  may be factorized as  
\[
S(z)=U(z)
(I-z^{-1}G),\quad   S_q(z)=U_q(z)(I-z^{-1}G),
\]
where $U(z)=I-\sum_{i=0}^\infty A_i^* z^i$,
$U_q(z)=I-\sum_{i=0}^q B_i^* z^i$, the matrices $A_i^*$, $i \geq 0$,
are defined in \eqref{eq:H} and $B_i^*=\sum_{j=i}^{q} A_j(G) G^{j-i}$,
$i=0,\ldots,q$.

Moreover, $B_0^*=A_0^*$,
$B_i^* \le A_i^*$ for  
$i=1,\ldots,q$, and, if 
$\sum_{i=-1}^\infty A_iz^{i+1}$ is a matrix polynomial
 and if the drift $\mu$ is negative,  there exists $\xi_q\ge \xi>1$ such that $\det S_q(\xi_q)=0$ and
$
\xi_q=
\min\{|z|:z\in\mathbb{C},|z|>1,\det S_q(z)=0\}.
$
\end{theorem}

\begin{proof}
  The factorization of $S(z)$ is known 
  see \cite[Theorem 4.13]{blm:book}.  To verify that
  $S_q(z)=(I-\sum_{i=0}^q B_i^* z^i)(I-z^{-1}G)$, we use the fact that
  $G$ is a solution of the matrix equation \eqref{eq:mateqq} and
  verify that the coefficients of equal powers of $z$ are equal.

To verify that $B_0^*=A_0^*$, we replace $z$ with $G$ in
\eqref{eq:alz}.  To verify that $B_i^*\le A_i^*$, $i=1,\ldots,q$, we
define the matrix power series $B_i^*(z)=\sum_{j=i}^q A_j(z)z^{j-i}$
  and $A_i^*(z)=\sum_{j=i}^\infty A_j z^{j-i}$ and show that
  $B_i^*(z)\le A_i^*(z)$. Since the coefficients of these power series and
  $G$ are nonnegative matrices, by replacing $z$ with $G$, that
  inequality implies $B_i^*\le A_i^*$, $i=1,\ldots,q$.  The inequality
  $B_i^*(z)\le A_i^*(z)$ is equivalent to
  $z^iB_i^*(z)\le z^iA_i^*(z)$ and we obtain from \eqref{eq:alz} that
 $
B_i^*(z)z^i- A_i^*(z)z^i=\sum_{j=0}^{i-1}A_jz^j-\sum_{j=0}^{i-1}A_j(z)z^j\le  0,
$
under the assumption  $A_\ell(z)\ge A_\ell$, $\ell=0,\ldots,q$.
By denoting $u(z)=\det U(z)$, if $\mu<0$ then the roots of $u(z)$ lie
outside the closed unit disk, $u(\xi)=0$ and
$\xi=\min\{|z|:z\in\mathbb{C},|z|>1,u(z)=0\}$. Moreover, since
$\det(I-A_0^*)\ne 0$, then $u(z)=0$ if and only if $\tilde u(z)=0$,
where $\tilde u(z)=\det(I-\sum_{i=1}^d C_iz^i)$, with
$C_i=(I-A_0^*)^{-1}A_i^*$ and $d$ is such that $A_i=0$ for $i>d$.
Similarly, by denoting $u_q(z)=\det U_q(z)$, since $B_0^*=A_0^*$, we
have $u_q(z)=0$ if and only if $\tilde u_q(z)=0$, where
$\tilde u_q(z)=\det(I-\sum_{i=1}^q \tilde C_iz^i)$, with
$\tilde C_i=(I-A_0^*)^{-1}B_i^*$, $i=1,\ldots,q$. The inequality
$\xi_q \geq \xi$ follows by applying Lemma \ref{lem:comp} to the
reversed matrix polynomials $z^d(I-\sum_{i=1}^d C_iz^{-i})$ and
$z^d(I-\sum_{i=1}^q \tilde C_iz^{-i})$.
\end{proof}

\subsection{Embedding the mass in the largest degree coefficient}\label{sec:emb}
It follows from Theorem \ref{thm:4.1} that the fastest convergence of
the sequence $\{ X_k\}_k$ is obtained by embedding the tail of the
series into the coefficient of largest degree. Moreover, embedding
the mass into a coefficient of index $q_2>q_1$, gives a sequence having a faster
convergence rate. Hence, the larger  $q$, the faster
convergence of the sequence $\{ X_k\}_k$.  In the limit case where
$A(z)$ is a polynomial of degree $d$, one iteration
is sufficient to obtain $G$ if one sets $q=d-1$. However, in this latter case, the new algorithm does not provide any advantage, since the equation to be solved coincides with the original one.
If the coefficients are defined as in \eqref{eq:mass}, one easily
checks from  Theorem~\ref{thm:fact} that $U_q(z)$ is obtained by
truncating the series $U(z)$ at a polynomial of degree $q$, i.e.,
$U_q(z)=I-\sum_{i=0}^{q} A_i^*z^i$.   Furthermore, it follows from
Lemma~\ref{lem:comp}, that $\xi_{q_1}\ge \xi_{q_2}$ if
$q_1<q_2$. Therefore, the numerical properties of the matrix equation
\eqref{eq:mateqq} with coefficients defined in \eqref{eq:mass} are
better for smaller values of $q$.
Hence, there is an optimal value $q$ which results from a trade-off
between the good convergence properties of the sequence $\{X_k\}_k$
and the good numerical properties of the matrix equation \eqref{eq:mateqq}
to be solved at each step $k$.  From a theoretical point of view, it
is difficult to determine the optimal value of $q$.  We will discuss
this issue in Section~\ref{sec:cost}.

\section{The case of stochastic initial approximation}
   \label{sec:stoc}
In this section we study the convergence of the sequence
\eqref{eq:fixquad2}, in the case where the starting approximation
$X_0$ is a stochastic matrix, and we prove that it is formed of
stochastic matrices and converges to the stochastic solution of
\eqref{eq:mateq}.

\begin{theorem}\label{thm:convst}
Assume that the drift $\mu$ is nonpositive and let $X_0$ be a stochastic matrix. Then, for any $k\ge 0$,
the matrix equation \eqref{eq:fixquad2} has a unique stochastic solution $X_{k+1}$, so that the sequence $\{X_k\}_k$ is well defined. Moreover, the sequence $\{X_k\}_k$ converges to the minimal nonnegative solution 
$G$ of \eqref{eq:mateq}, which is stochastic.
\end{theorem}
\begin{proof}
We prove by induction  that $X_k$ is stochastic. For $k=0$, $X_k$ is stochastic. Assume that, for a $k\ge 0$, the matrix $X_k$ is stochastic. 
Observe that $A_i(X_k)\ones=A_i(1)\ones$ so that 
from \eqref{eq:sum2} it follows that $\sum_{i=-1}^q A_i(1)\ones =\ones$.
Thus, applying Theorem \ref{thm:existsto}, we obtain that
\eqref{eq:fixquad2} has a unique stochastic solution $X_{k+1}$. 
Since $\mu\le 0$ then $G$ is stochastic. 
We prove that $\lim_k X_k=G$. Observe that stochastic matrices
form a compact set so that the sequence $\{X_k\}_k$ has a converging
subsequence $\{X_{k_i}\}_i$ which converges to a stochastic matrix $S$. Consider the sequence  defined by recursion \eqref{eq:fixquad2}, obtained by starting with the null matrix, and denote such sequence by $\{ Y_k\}_k$.
 We may easily show by induction that $Y_k\le X_k$ for any $k\ge 0$. Since $\lim_{k\to\infty}Y_k=G$, then $G\le S$.
 Since both $G$ and $S$ are stochastic, then $G=S$.
Therefore, any  converging subsequence of $\{X_k\}_k$
converges to the same limit $G$, therefore the sequence
$\{X_k\}_k$ is convergent and converges to $G$.
\end{proof}

Now we will show that, if $\mu\le 0$, so that $G$ is stochastic, then the
sequence obtained with $X_0$ stochastic converges faster than the
sequence obtained with $X_0=0$.  To this aim, we need to
rewrite \eqref{eq:err1} in a slightly different way.  We subtract
\eqref{eq:fixquad2} from the equation
$G=\sum_{\ell=-1}^q A_\ell(G)G^{\ell+1}$ and obtain
$
E_{k+1}= 
\sum_{\ell=-1}^q 
(A_{\ell}(X_k)G^{\ell+1}-X_{k+1}^{\ell+1})+
(A_\ell(G)-A_\ell(X_k))G^{\ell+1}),
$
where $E_k=G-X_k$.  We use Lemma \ref{t:tech} and find that
\begin{equation}\label{eq:err2}
\begin{aligned}
&E_{k+1}=\sum_{\ell=0}^qA_\ell(X_k)\sum_{j=0}^\ell X_{k+1}^j
E_{k+1}G^{\ell-j} +\widehat S(E_k),\\ 
&\widehat S(E_k)=\sum_{\ell=-1}^q
\sum_{i=1}^\infty A_{\ell,i} \sum_{j=0}^{i-1}X_k^j E_k G^{i-j+\ell}.
\end{aligned}
\end{equation}
We write the matrix product $Y=AXB$ as $\ve(Y)=(B\tr\otimes A)\ve(X)$,
where $\otimes$ is the Kronecker product and $\ve(C)$ is the vector
obtained by stacking the columns of the matrix $C$.  
Setting $\eta_k=\ve(E_k)$, we rewrite \eqref{eq:err2} as
\begin{equation}\label{eq:QGG}
(I-Q(X_k,X_{k+1}))\eta_{k+1}=P(X_k)\eta_k,~~k=0,1,\ldots,
\end{equation}
where
\begin{equation}\label{eq:QGG2}
\begin{aligned}
&Q(X,Y)=
\sum_{\ell=0}^q\sum_{s=0}^\ell \left( (G^{s})\tr\otimes
 A_\ell(X)Y^{\ell-s} \right), \\
&P(X)=\sum_{\ell=-1}^q\sum_{i=1}^\infty \sum_{s=\ell+1}^{\ell+i} (G^{s})\tr\otimes (A_{\ell,i}X^{i+\ell-s}).
\end{aligned}.
\end{equation}

As we are interested in asymptotic convergence results, we analyse in
the next lemma the spectral properties of the matrices $Q(G,G)$ and
$P(G)$.  

\begin{lemma}\label{lem:18} 
  Let $\lambda_1,\ldots,\lambda_m$ be the eigenvalues of $G$. The
  set of eigenvalues of the matrix $Q(G,G)$, defined in
  \eqref{eq:QGG2}, is the union of the sets of eigenvalues of the
  matrices
  $\sum_{\ell=0}^q\sum_{s=0}^\ell \lambda_i^sA_\ell(G)G^{\ell-s}$ for
  $i=1,\ldots,m$. In particular,
  $\rho(Q(G,G))\le \rho( \sum_{\ell=0}^q\sum_{s=0}^\ell
  A_\ell(G)G^{\ell-s})$
  and, if the drift $\mu$ of \eqref{eq:mu} is negative, then
  $\rho(Q(G,G)) <1$, so that $I-Q(G,G)$ is invertible.
\end{lemma}
\begin{proof}
  Let $T=SG\tr S^*$ be the Schur form of the matrix $G\tr$, where $T$
  is upper triangular with diagonal entries
  $\lambda_1,\ldots,\lambda_m$, $S$ is a unitary matrix and the symbol $*$ denotes conjugate transposition. The matrix
$
(S\otimes I)Q(G,G)(S^*\otimes I)
$
 is block upper triangular with diagonal blocks $\sum_{\ell=0}^q\sum_{s=0}^\ell \lambda_i^sA_\ell(G)G^{\ell-s}$ for $i=1,\ldots,m$.
Since the set of eigenvalues of a block triangular matrix is the union
of the sets of  eigenvalues of the diagonal blocks, then the first
claim follows. Moreover,
as $|\lambda_i|\le 1$ for any $i$, we have
$
|\sum_{\ell=0}^q\sum_{s=0}^\ell \lambda_i^sA_\ell(G)G^{\ell-s}|\le
\sum_{\ell=0}^q\sum_{s=0}^\ell |\lambda_i|^sA_\ell(G)G^{\ell-s}\le
\sum_{\ell=0}^q\sum_{s=0}^\ell A_\ell(G)G^{\ell-s}
$.
If $\mu <0$, the right-most matrix in the inequality above has
spectral radius less than 1 in view of Theorem~\ref{th:12}. Therefore,
$\rho(Q(G,G))<1$  by Theorem~\ref{thm:pf}.
\end{proof}

If the drift $\mu$ is negative, it follows from the invertibility of
$I-Q(G,G)$ that, if the sequence $\{X_k\}_k$ converges to $G$, there
exists $k_0>0$ such that for any $k\ge k_0$ the matrix
$I-Q(X_k,X_{k+1})$ is invertible and, from \eqref{eq:QGG}, we may
write
\begin{equation}\label{eq:etak}
\eta_{k+1}=(I-Q(X_k,X_{k+1}))^{-1}P(X_k)\eta_k.
\end{equation}

If $X_0=0$, since the sequence  $\{X_k\}_k$ converges monotonically to $G$, then $\eta_k\ge 0$ for any $k$ and
$(I-Q(X_k,X_{k+1}))^{-1}P(X_k)\le 
(I-Q(G,G))^{-1}P(G)$. Therefore, 
$
\eta_k\le W^k \eta_0
$, 
where $W=(I-Q(G,G))^{-1}P(G)$, and
$
\| \eta_k\|\le \| W^k\| \| \eta_0\|
$
for any operator norm $\| \cdot \|$,
so that the asymptotic rate of convergence is
\begin{equation}\label{eq:r0}
r^{(0)}=\limsup_k \left(\frac{\|\eta_k\|}{\|\eta_0\|}\right)^\frac1k \le \lim_k \| W^k\|^\frac1k = \rho(W).
\end{equation}

To study the spectral properties of the matrix $W$, we follow an argument similar to the one used in the proof of Lemma \ref{lem:18}.
Since $\ones\tr G\tr=\ones\tr$, we may find 
a unitary matrix $S$, having as first row
$\frac{1}{\sqrt{m}}\ones\tr$, such that
 $T=SG\tr S^*$ is a Schur form of $G\tr$. With this choice,
  the diagonal entries of $T$ are $1,\lambda_2,\ldots,\lambda_m$.

Define $\mathcal P=S\otimes I$. We may verify that the matrix
$\mathcal P W \mathcal P^*$ is a block upper triangular matrix of the form
\begin{equation}\label{eq:t1t2}
\mathcal P W \mathcal P^*=\left[ \begin{array}{cc}
T_1 & * \\
0 & T_2
\end{array} \right],
\end{equation}
where $T_1=M^{-1}N$, with $M$ and $N$ defined in \eqref{eq:MN2}, and  $T_2$ has size $(m^2-m)\times (m^2-m)$. By following the same arguments used in the proof of Lemma \ref{lem:18}, we may show that $\rho(T_2)\le \rho(T_1)$, so that $\rho(W)=\rho(M^{-1}N)$. In particular equation \eqref{eq:r0} provides the same bound $r \le\rho(M^{-1}N)$ obtained in Section \ref{sec:convq}.

If $X_0$ is a stochastic matrix, then  $\{X_k\}_k$ is a sequence of
stochastic matrices that converges to $G$. Therefore $E_k\ones =0$
and  $\eta_k =\ve(E_k)$  belongs to the subspace  orthogonal to the
vectors of the form $\ones\otimes v$ for any $v\in \mathbb R^m$,
and the vector $s_{k}=\mathcal P \eta_k$
 has its first $m$ entries equal to zero, i.e., $s_k\tr=[0,\ldots,0,\hat s_k\tr]$, where $\hat s_k$ has size $m^2-m$.
Since $X_k$ is stochastic, the first column of $SX_kS^*$ is the first column of the identity matrix.
Therefore, defining $W_k=(I-Q(X_k,X_{k+1}))^{-1}P(X_k)$, we have
\begin{equation}\label{eq:php}
\mathcal P W_k \mathcal P^*=\left[ \begin{array}{cc}
T_{1,k} & * \\
0 & T_{2,k}
\end{array} \right],
\end{equation}
where $T_{1,k}$ is an $m\times m$ matrix and $T_{2,k}$ is $(m^2-m)\times (m^2-m)$.
From 
 \eqref{eq:etak} and \eqref{eq:php}, we conclude that
$
\hat s_{k+1}=T_{2,k} \hat s_k$.

Since the asymptotic rate of convergence is independent of the norm,  we may choose the norm $\|x\|':=\|\mathcal P x\|_\infty$. Therefore $\| \eta_k\|'=\|\hat s_k\|_\infty$ and the asymptotic rate of convergence in the stochastic case is
$
r^{(\hbox{\scriptsize sto})}=\limsup_k \| \hat s_k\|_\infty^{1/k}=
\limsup_k\sigma_k$, $\sigma_k= \| T_{2,k-1}T_{2,k-2}\cdots T_{2,k_0}\hat s_{k_0}\|_\infty^{1/k},
$
where $k_0$ is such that $\det(I-Q(X_k,X_{k+1}))\ne0$ for any $k\ge k_0$. 
Since 
$\sigma_k
\le \| T_{2,k-1}T_{2,k-2}\cdots T_{2,k_0}\|_\infty \|\hat s_{k_0}\|_\infty
$
 and $\lim_{k\to\infty}T_{2,k}=T_2$, we have
$
r^{(\hbox{\scriptsize sto})}\le \lim_{k\to\infty}\|T_{2,k}\|_\infty^{1/k}=\rho(T_2),
$
 where the latter implication follows by the same arguments as 
used in \cite{meini:fi}. 
We may conclude with the following theorem.

\begin{theorem}
  Let $r^{(0)}$ and $r^{(\hbox{\scriptsize\rm sto})}$ be the asymptotic
  rates of convergence of the sequences \eqref{eq:fixquad2} with
  $X_0=0$, and with $X_0$ equal to a stochastic matrix, respectively.  Then
  $r^{(0)}\le\rho(W)$. If the drift $\mu$ of \eqref{eq:mu} is
  negative, then $r^{(\hbox{\scriptsize\rm sto})}\le\rho(T_2)$, where
  $T_2$ is the matrix in \eqref{eq:t1t2}. If $W$ is irreducible and
  aperiodic, then $\rho(T_2)<\rho(W)$, otherwise the weak inequality
  holds.
\end{theorem}

\section{Computational cost and stability analysis}\label{sec:cost}

In the analysis of the computational cost, we assume that the power series $A(z)$ is a matrix polynomial of degree $d$.
We look at the proposed method as a
two-level iterative method, where the {\em outer iteration} is the iteration defined by \eqref{eq:fixquad2}, while the {\em inner iteration} is the iteration applied to solve the matrix equation of degree $q+1$ at each step $k$ of the outer iteration. 
At each step of the outer iteration, we have to compute the  coefficients of the matrix equation  \eqref{eq:fixquad2}. 
In this analysis we restrict the attention to the case where the mass is embedded in the coefficient $A_q$ of the term of degree $q+1$, i.e., the coefficients are defined by \eqref{eq:mass}.
From \eqref{eq:fixquad2} and \eqref{eq:mass}, at each step of the outer iteration, we have to compute 
$A_q(X_{k})=\sum_{i=0}^{d-q-1}A_{q+i}X_{k}^i$. By using Horner's rule, the cost of the computation of $A_q(X_{k})$ is $2m^3(d-q-1)$ arithmetic operations, where we neglect the $O(m^2)$ terms.

The computational cost of the inner iterations depends on the numerical method used to solve the matrix equation of degree $q+1$. By applying the $U$-based functional iteration, we generate the sequence
\begin{equation}\label{eq:inn}
Z_{\nu+1}=\left( I-\sum_{i=0}^{q-1} A_i Z_\nu^i - A_q(X_k)Z_\nu^q \right)^{-1}A_{-1},\qquad\nu=0,1,\ldots.
\end{equation}
If the matrix inversion is performed by computing the $LU$ factorization and by solving the linear systems, the computational cost per step is $2m^3(q+4/3)$ arithmetic operations, where we neglect the $O(m^2)$ terms.

Therefore, by denoting $N_{\hbox{\scriptsize out}}$ and $N_{\hbox{\scriptsize in}}$  the number of outer iterations and the overall number of inner iterations, the computational cost is
$2m^3(N_{\hbox{\scriptsize out}} (d-q-1)+ N_{\hbox{\scriptsize in}}(q+4/3))$ arithmetic operations.
This estimate should be compared with  the cost of the $U$-based iteration, which is $2m^3N_{\hbox{\scriptsize $U$-based}}(d+1/3)$, where $N_{\hbox{\scriptsize $U$-based}}$ is the number of iterations.

As shown in Section~\ref{sec:opt}, the number of outer iterations decreases as $q$ increases. Indeed, if $A_i=0$ for $i\ge d$, in the limit case of $q=d-1$, one outer iteration is enough to compute the solution $G$.
On the other hand, as pointed in Section~\ref{sec:emb},
smaller values of $q$ provide larger values of $\xi_q$.
This properties implies that, with smaller values of $q$ the conditioning of the matrix equation \eqref{eq:mateqq} is better, and the number of inner iterations for \eqref{eq:mateqq} is lower. 

 Hence, there is a trade-off between the good convergence properties of $\{X_k\}_k$ and good numerical properties of the matrix equation \eqref{eq:mateqq}.
 It is difficult to determine the optimal value of $q$. However, this
 is an asymptotic analysis and, in practice, the number of inner iterations strongly depends on the starting approximation and on the stop condition.
 
In the numerical experiments presented in Section \ref{sec:num}
we have halted the outer iteration 
if the residual error in the infinity norm, that is 
$\delta_k=\frac 1m \|X_k-\sum_{i=0}^dA_{i-1}X_k^i\|_\infty$, is less than $\epsilon=10^{-15}$ 
or if $\delta_{k}$ is significantly larger than the error at the previous step, i.e.,
$\delta_{k}>\delta_{k-1}(1+10^{-3})$. 
At the $k$-th outer step,  we choose $Z_0=X_{k}$ as starting approximation of the inner iteration \eqref{eq:inn} for computing $X_{k+1}$.
The inner iteration is stopped 
if the residual error 
$\tilde{\delta}_\nu=\frac1m \| Z_{\nu}-\sum_{i=0}^{q} A_{i-1} Z_\nu^{i} - A_q(X_k)Z_\nu^{q+1}\|_\infty$ satisfies the condition
$
\tilde{\delta}_\nu<\max\left\{\frac{1}{10}\delta_k, 4u,\frac14 \epsilon\right\},
$
where $u$ is the machine precision, 
or if it is significantly larger than the error at the previous step, i.e, $\tilde\delta_{\nu}>\tilde\delta_{\nu-1}(1+10^{-3})$. 
 As a consequence of this choice, the number of inner iterations does not grow as $q$ grows (see Figure \ref{fig:h}). In all the numerical experiments,  the optimal value of $q$, in terms of overall CPU time, is generally much smaller than $d$.

\section{Numerical experiments}\label{sec:num}
In this section we report  some numerical experiments which validate the theoretical results obtained in the previous sections and show the improvement of the computational efficiency of the new fixed point iterations with respect to classical iterations. All the algorithms have been implemented in Matlab and tested on a Laptop i3-7100 CPU 3.90GHz$\times$4.

\subsection{Test problems}

We have considered two kinds of test problems. The tests of the first kind are
generated synthetically in such a way that all the eigenvalues of the
matrix $G$ have modulus close to 1.
The tests of the second kind are PH/PH/1 queues see \cite{hn98,lr97b}.
We provide below a description of these two classes of problems.\medskip

\subsubsection{Synthetic examples}
We have generated an M/G/1-type Markov chain associated with the matrix polynomial $A(z)=\sum_{i=0}^{d} A_{i-1} z^{i}$ of degree $d$, where the matrix coefficients have size $m\times m$. The matrix coefficients have been constructed in such a way that the drift $\mu$  of the Markov chain \eqref{eq:mu} is close to a given negative value.

Let $C=(c_{i,j})$ be the $m\times m$ circulant matrix such that $c_{i,j}=1$ if $j-i\equiv 1$ mod $m$,  $c_{i,j}=0$ elsewhere.
Let
$v_i$, $i=-1,\ldots,d-1$, be nonnegative real numbers such that $\sum_{i=0}^d v_{i-1}=1$. Since $CC\tr =I$, if $A_i=v_iC^i$ for $i=-1,0,\ldots,d-1$, then $G=C\tr $ solves the equation $G=\sum_{i=0}^d A_{i-1}G^i$. Moreover, since $\ones\tr C^i=\ones\tr $, and $\sum_{i=0}^d v_{i-1}=1$ then 
 $\ones\tr \sum_{i=0}^d A_{i-1}=\ones\tr $ so that
the drift \eqref{eq:mu}  is $\mu=
-v_{-1}+\sum_{i=1}^{d-1}iv_i$.

We use the above properties to generate matrix coefficients  in such a way that the drift is close to an assigned negative value. More specifically, given $\mu<0$, $0<s_1,s_2<1$, and a small positive number $\sigma$, we define 
$
\tilde A_i=v_i C^i+\sigma s_2^{m(i+1)}R_i \Delta,~~i=-1,\ldots,d-1,
$
 where $R_i$ is a random $m\times m$ matrix with entries uniformly distributed between 0 and 1, $\Delta$ is the diagonal matrix with diagonal entries $1,s_2,\ldots,s_2^{m-1}$, and $v_{-1}=\frac{1-s_1^{d-1}}{1-s_1}-\mu$, $v_0=1-v_{-1}-\sum_{i=1}^{d-1}\frac{s_1^{i-1}}{i}$, $v_i=\frac{s_1^{i-1}}{i}$ for $i=1,\ldots,d-1$.
The basis $s_1$ cannot be chosen too close to 1, otherwise $v_0$ is negative.
We may easily check that $\sum_{i=-1}^{d-1}v_i=1$ and that $-v_{-1}+\sum_{i=1}^{d-1}iv_i=\mu$. 
Therefore, if $\sigma=0$, then the drift is exactly $\mu$ and $G=C\tr $ is the minimal nonnegative solution of \eqref{eq:mateq}. If $\sigma>0$ the matrix $\sum_{i=-1}^{d-1}\tilde A_i$ is not stochastic, therefore we define $A_i=D^{-1}\tilde A_i$, $i=-1,\ldots,d-1$, where $D$ is the diagonal matrix with diagonal entries equal to the components of the  vector $\sum_{i=-1}^{d-1}\tilde A_i\ones$.
If $\sigma>0$ is a small number, then the drift is close to the given value $\mu$ and the minimal nonnegative solution $G$ is a small perturbation of $C\tr $. Since the eigenvalues of $C$ are the $m$-th roots of 1, then all the eigenvalues of $G$, except the eigenvalue equal to 1,  have modulus close to 1.
This latter property increases the difficulty of the computation of $G$ see \cite{blm:book}. 

In our experiments,   
we have chosen size $m=20$, degree $d=1500$, and drift $\mu\in\{-0.1, -0.05,-0.01,$ $-0.005\}$.  The two bases $s_1$ and $s_2$ of the exponential decay of the coefficients have been chosen 
as $s_1=0.6$ and  $s_2=0.9995$, while the parameter $\sigma$ for the random perturbation has been chosen as $\sigma=10^{-11}$. We recall that $\mu=0$ means that the Markov chain is null recurrent, in this case the problem is more difficult from the computational point of view, in fact,  the convergence of the fixed point iterations slows down and the problem is more ill-conditioned.\medskip

\subsubsection{PH/PH/1 queues}
We briefly recall the definition of PH/PH/1 queues. For a detailed description we refer the reader to \cite{hn98,lr97b}.
Consider two sequences $\{X_h\}_h$ and $\{Y_h\}_h$ of independent continuous random variables with PH($\vtau,T$)  and
PH($\vbeta,S$) distributions, respectively. Here $\vtau$ and $\vbeta$ are probability vectors of length ${n_1}$ and ${n_2}$, respectively, while $T$ and $S$ are subgenerators of size ${n_1}\times {n_1}$ and ${n_2}\times {n_2}$, respectively, i.e., $-T$ and $-S$ are nonsingular M-matrices.
In the queueing
applications,  $\{X_h\}_h$ represents the intervals between successive
arrivals and $\{Y_h\}_h$ represents the service durations.
Assume points are marked on
a time axis at the epochs $X_1+X_2+\cdots +X_h$ and at the epochs
$Y_1+Y_2+\cdots +Y_h$, $h \geq 1$.

Let $(A_{h-1})_{ij}$ be the probability that $h$ points of type $X$
occur in an interval of type $Y$, and the phase of the last interval
of type $X$ is
$j$, given the phase of the first interval is $i$; the first phase of the
interval $Y$ has distribution $\vbeta$.  In queueing applications,
these would be the probabilities that the queue increases by $h$ units
during a service interval.
Define
$
M_1  = -( T \otimes I_{n_1}+I_{n_2}\otimes S)^{-1} (\vt \cdot \vtau\tr \otimes I_{n_2})$,
$M_0  = -( T \otimes I_{n_1} + I_{n_2} \otimes S)^{-1} (I_{n_1} \otimes \vs \cdot \vbeta\tr)$,
where $\vt = - T \vone$ and $\vs=-S \vone$.
We have
\begin{equation}\label{eq:phph}
A_h = (I_{n_1} \otimes \vbeta\tr) M_1^{h+1} M_0 (I_{n_1}\otimes \vone) 
\qquad \mbox{for all $h \geq -1$.}
\end{equation}
The above matrices are nonnegative and their sum is stochastic if  $-\vbeta\tr S^{-1} \vone < -\vtau\tr T^{-1} \vone$. 
In our experiments, we have chosen as PH($\vbeta,S$) an Erlang
distribution, see \cite{rl89b}.  We start from an Er(${n_2},\lambda$) distribution with ${n_2}=10$ phases and $\lambda=10$, that
is,
$
\vbeta\tr = \vligne{1 & 0 & \ldots & 0}$,
$S=(s_{i,j}), s_{i,i}=-\lambda$, $s_{i,i+1}=\lambda$, $s_{i,j}=0$ elsewhere. 
This matrix is such that $-\vbeta\tr S^{-1} \vone =
1$. 

For the PH($\vtau,T$) distribution, we have taken a pseudo
heavy-tailed distribution, which is used in \cite{dls18} and borrowed
from \cite{rl97b}.  Define the transition matrix $Q=(q_{i,j})$ such that 
$q_{1,1}=-(c+s_a)$, $q_{i,1}=q_{i,i}=-(b/a)^{i-1}$ for $i=2,\ldots,n_2$,
$q_{1,i}=(1/a)^{i-1}$, $i=2,\ldots,n_2$, $q_{i,j}=0$ elsewhere,
where $s_a = (1/a) +(1/a)^2 + (1/a)^3 + \cdots + (1/a)^{{n_2}-1}$.  
The parameters must satisfy the conditions
$a >1$, $a>b>0$, $c >0$.  The initial probability vector is
$\vtau\tr = \vligne{1 & 0 & \cdots & 0}$ and the matrix $T_0$, defined as 
$T_0= (-\vtau\tr Q^{-1} \vone) Q$,  is  such that $-\vtau\tr T_0^{-1} \vone = 1$.
We have chosen the values $a=2$, $b=1$, $c=1.5$ and ${n_1}=10$.
In order to have the expected interval between arrivals equal to
$1/\rho$, we take $T=\rho T_0$.
In summary, the two distributions are  normalized in such a way that
$-\vbeta\tr S^{-1} \vone < -\vtau\tr T^{-1} \vone=1/\rho$ so that
the queue is stable if $\rho < 1$.
Moreover, we may verify that the drift of the M/G/1-type Markov chain
defined by the matrices \eqref{eq:phph} is $\mu=1-\rho$.
The value of $\rho$ has been taken to be $\rho=0.85$.

The matrix power series obtained this way has blocks of size $10\times 10$, and has been truncated to a matrix polynomial of degree $d=61$ so that the infinity norm of the remainder is less than $10^{-16}$.

\subsection{Numerical results}

We have performed different kinds of tests, where the sequences generated by all the fixed point iterations have been started either with 
$X_0=0$ or with $X_0=I$. 
In its wider generality, we generate the sequence \eqref{eq:fixquad2}, where at each step $k$ we solve a matrix equation of degree $q+1$.  
This way we obtain a two-level iterative method, where the {\em outer iteration} is the iteration defined by \eqref{eq:fixquad2}, while the {\em inner iteration} is the iteration applied to solve the matrix equation of degree $q$ at each step $k$ of the outer iteration. \medskip

\subsubsection{The synthetic examples.}
For the synthetic case, as starting approximation for the tested fixed
point iterations, we have always chosen $X_0=I$. Indeed, since all the
eigenvalues of the matrix $G$ are close to one, the performances of
the fixed point iterations are not much different if we start with the null matrix or with a stochastic matrix.

The first test aims to compare the convergence speed of the $U$-based
iteration and of the three iterations obtained by setting $q=1$ and
embedding the tail of the matrix polynomial into the constant, the
linear and the quadratic coefficient, respectively.

In Figure \ref{fig:1} we report the semi-logarithmic plot of the residual error 
in the infinity norm $e_k=\frac1m \|X_k-\sum_{j=0}^d A_{j-1}X_k^{j}\|_\infty$, where $m$ is the matrix size,
for the three iterations together with the residual error of the
$U$-based iteration; to the left the case with drift $\mu=-0.1$, to
the right the case where $\mu=-0.005$. As we can see from this plot,
the four graphs have different slopes, in  accordance with Theorems
\ref{thm:conf} and \ref{th:9}.
In particular, for $\mu=-0.005$, the number of steps needed to have a residual error less than $\epsilon=10^{-15}$ for the $U$-based iteration and for the three iterations relying on the solution of the quadratic equation is 
2170, 1778, 1325, 877, respectively.  
This shows a substantial improvement of our approach in terms of convergence speed.

\begin{figure} 
        \begin{center}
        \begin{tikzpicture}[scale=0.50]
        \begin{semilogyaxis}[
                legend pos = north east, width = .8\linewidth, 
                xlabel = {Iterations}, 
                ylabel = {Residual}, title = {$\mu = -0.1$}]
            \addplot[mark=o, color=blue] table {vresu_1.dat}; 
            \addplot[mark=triangle, color=red] table {vresn1_1.dat}; 
            \addplot[mark=square, color=green] table {vres0_1.dat}; 
            \addplot[mark=diamond, color=cyan] table {vres1_1.dat}; 
\legend{$U$-based, mass on $A_{-1}$, mass on $A_0$, mass on $A_1$}
        \end{semilogyaxis}
        \end{tikzpicture}
~        
        \begin{tikzpicture}[scale=0.50]
        \begin{semilogyaxis}[
                legend pos = north east, width = .8\linewidth, 
                xtick={0,400,800,1200,1600,2000},
                xlabel = {Iterations}, 
                ylabel = {Residual}, title = {$\mu = -0.005$}]
            \addplot[mark=o, color=blue] table {vresu_4.dat}; 
            \addplot[mark=triangle, color=red] table {vresn1_4.dat}; 
            \addplot[mark=square, color=green] table {vres0_4.dat}; 
            \addplot[mark=diamond, color=cyan] table {vres1_4.dat}; 
\legend{$U$-based, mass on $A_{-1}$, mass on $A_0$, mass on $A_1$}
        \end{semilogyaxis}
        \end{tikzpicture}
\end{center}\caption{\footnotesize
Residual errors, for two values of the drift $\mu$, of the three iterations obtained by embedding the tail of the series into the constant, linear and quadratic term, respectively, and of the $U$-based iteration. }\label{fig:1}
\end{figure}
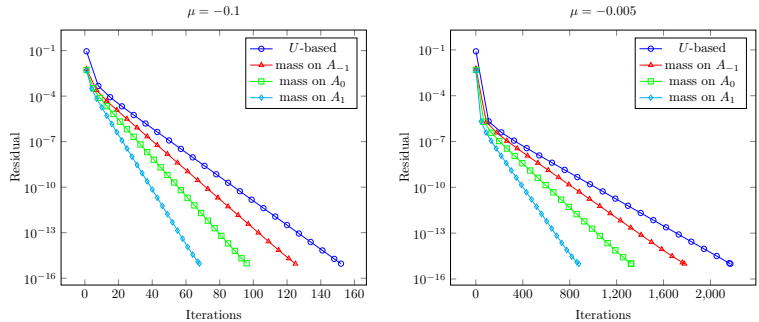

The second test aims to compare the effective gain that we have in terms of CPU time by using the $U$-based iteration or the iteration obtained by embedding the tail in the constant, linear, and quadratic term. In this case, we may have different possibilities according to the way in which the quadratic equation is solved at each step (inner iteration). We have considered two different implementations for this resolution which differ by the way the quadratic equation is solved. The first implementation simply applies the $U$-based iteration as inner iteration, the second implementation applies Cyclic Reduction (CR) as inner iteration.

In Table \ref{tab:1} we report the speed--up of the CPU time with
respect to the $U$-based iteration. For different values of $\mu$, we
report in the first three columns the values of the speed-up obtained
by embedding the mass in $A_{-1}$, $A_0$ and $A_1$, respectively,
where the quadratic equation is solved by means of inner $U$-based
iterations. The last column corresponds to the case where the mass is
embedded in $A_1$ and the quadratic equation is solved by means of the
CR algorithm.
 We see that the acceleration in terms of CPU time obtained by the combination of our algorithm with CR is by a factor greater than 2.
\begin{table} \scriptsize 
\begin{center}
\begin{tabular}{c|cccc}
$\mu ~\backslash~$Alg  &\tt $A_{-1}$  &\tt $A_0$ &\tt $A_1$& $A_1$-CR \\ \hline
$-0.1$& 1.2 & 1.4 & 2.1 & 2.3 \\ 
$-0.05$& 1.2 & 1.4 & 2.0 & 2.3 \\ 
$-0.01$& 1.2 & 1.5 & 2.2 & 2.7 \\ 
$-0.005$& 1.2 & 1.5 & 2.3 & 2.8 \\ 
$-0.001$& 1.2 & 1.5 & 2.5 & 2.8 \\ 
$-0.0005$& 1.2 & 1.5 & 2.3 & 2.9 \\ 
$-0.0001$& 1.1 & 1.5 & 2.2 & 3.0 \\ 
\end{tabular}\caption{\footnotesize
Speed--up, in terms of the CPU time, with respect to the $U$-based iteration, of the algorithms obtained by embedding the mass in the coefficient $A_{-1}$, $A_0$, and $A_1$, respectively, where the quadratic equation  is solved by means of the $U$-based iteration. In the last column the quadratic equation is solved by means of Cyclic Reduction with the mass embedded in $A_1$.}\label{tab:1}
\end{center}
\end{table}
A substantially larger speed--up can be obtained by embedding the mass in the coefficients of higher degree terms. This computational analysis is performed in the next test.

In the third test, we implemented the algorithm where 
 at each step (outer iteration) a matrix equation of degree $q+1$ is solved, and 
 the tail of the matrix polynomial is embedded into the coefficient $A_q$ of the term of degree $q+1$. The matrix equation at each outer iteration is solved by means of the $U$-based algorithm as inner iteration, where the starting approximation is the current approximation of the outer iteration. For any value of $q+1$ in the range $[2,30]$ we computed the CPU time needed to arrive at a residual error less than $10^{-15}$ together with the number of outer iterations and the overall number of inner iterations. The graphs with these values are reported in Figure \ref{fig:h}.
\begin{figure} 
        \begin{center}
        \begin{tikzpicture}[scale=0.4]
        \begin{axis}[
                xlabel = {Degree $q+1$}, 
                ylabel = {CPU time}]
            \addplot[mark=o, color=blue] table {T1.dat}; 
            \addplot[mark=diamond, color=red] table {Tu1.dat}; 
            \addplot[mark=square, color=green] table {Tcr1.dat}; 
        \end{axis}
        \end{tikzpicture}   
~
        \begin{tikzpicture}[scale=0.4]
        \begin{axis}[
                xlabel = {Degree $q+1$}, 
                ylabel = {Outer Iterations}]
            \addplot[mark=o, color=blue] table {Oit1.dat};   
        \end{axis}
        \end{tikzpicture} 
~
        \begin{tikzpicture}[scale=0.4]
        \begin{axis}[
                xlabel = {Degree $q+1$}, 
                ylabel = {Inner Iterations}]
            \addplot[mark=o, color=blue] table {Iit1.dat}; 
            \addplot[mark=*, color=red] table {itu1.dat};  
        \end{axis}
        \end{tikzpicture} 
\\
        \begin{tikzpicture}[scale=0.4]
        \begin{axis}[
                xlabel = {Degree $q+1$}, 
                ylabel = {CPU time}]
            \addplot[mark=o, color=blue] table {T2.dat}; 
            \addplot[mark=diamond, color=red] table {Tu2.dat}; 
            \addplot[mark=square, color=green] table {Tcr2.dat}; 
        \end{axis}
        \end{tikzpicture}   
~
        \begin{tikzpicture}[scale=0.4]
        \begin{axis}[
                xlabel = {Degree $q+1$}, 
                ylabel = {Outer Iterations}]
            \addplot[mark=o, color=blue] table {Oit2.dat};   
        \end{axis}
        \end{tikzpicture} 
~
        \begin{tikzpicture}[scale=0.4]
        \begin{axis}[
                xlabel = {Degree $q+1$}, 
                ylabel = {Inner Iterations}]
            \addplot[mark=o, color=blue] table {Iit2.dat};   
            \addplot[mark=*, color=red] table {itu2.dat};
        \end{axis}
        \end{tikzpicture} 
\\
        \begin{tikzpicture}[scale=0.4]
        \begin{axis}[
                xlabel = {Degree $q+1$}, 
                ylabel = {CPU time}]
            \addplot[mark=o, color=blue] table {T3.dat}; 
            \addplot[mark=diamond, color=red] table {Tu3.dat}; 
            \addplot[mark=square, color=green] table {Tcr3.dat}; 
        \end{axis}
        \end{tikzpicture}   
~
        \begin{tikzpicture}[scale=0.4]
        \begin{axis}[
                xlabel = {Degree $q+1$}, 
                ylabel = {Outer Iterations}]
            \addplot[mark=o, color=blue] table {Oit3.dat};   
        \end{axis}
        \end{tikzpicture} 
~
        \begin{tikzpicture}[scale=0.4]
        \begin{axis}[
                xlabel = {Degree $q+1$}, 
                ylabel = {Inner Iterations}]
            \addplot[mark=o, color=blue] table {Iit3.dat};   
            \addplot[mark=*, color=red] table {itu3.dat};
         \end{axis}
        \end{tikzpicture} 
\\
        \begin{tikzpicture}[scale=0.4]
        \begin{axis}[
                xlabel = {Degree $q+1$}, 
                ylabel = {CPU time}]
            \addplot[mark=o, color=blue] table {T4.dat}; 
            \addplot[mark=diamond, color=red] table {Tu4.dat}; 
            \addplot[mark=square, color=green] table {Tcr4.dat}; 
        \end{axis}
        \end{tikzpicture}   
~
        \begin{tikzpicture}[scale=0.4]
        \begin{axis}[
                xlabel = {Degree $q+1$}, 
                ylabel = {Outer Iterations}]
            \addplot[mark=o, color=blue] table {Oit4.dat};   
        \end{axis}
        \end{tikzpicture} 
~
        \begin{tikzpicture}[scale=0.4]
        \begin{axis}[
                xlabel = {Degree $q+1$}, 
                ylabel = {Inner Iterations}]
            \addplot[mark=o, color=blue] table {Iit4.dat};   
            \addplot[mark=*, color=red] table {itu4.dat};
                    \end{axis}
        \end{tikzpicture} 
\end{center}\caption{\footnotesize
CPU time, in seconds, and number of inner / outer iterations required to compute $G$ within a residual error less than {\tt 1.e-15} by embedding the tail into the coefficient of degree $q+1$.  The equation of degree $q+1$ is solved by means of the $U$-based iteration. The two lines denoted by a green square and a red diamond,  mark the CPU time needed by PWCR \cite{blm:book} 
and by the $U$-based iteration, respectively.
The line denoted by a red star, in the third column, denotes the number of iterations needed by the $U$-based method.
From top to bottom the problems with drift
$-0.1$, $-0.05$, $-0.01$ and $-0.005$, are considered.} \label{fig:h}
\end{figure}
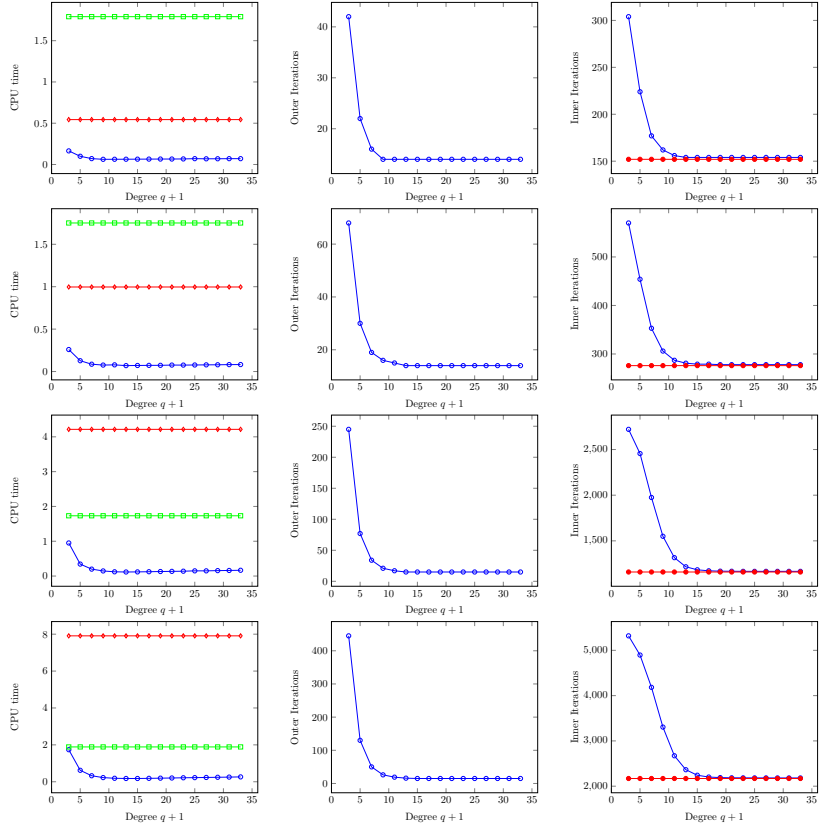

As we can see  from this figure, by increasing the value of $q$, we
obtain a rapid decrease of the CPU time. After reaching the minimum
value, the time slowly increases. It is also interesting to observe
that in this case, the optimal value of $q$ is less than 20. This
value is rather small with respect to the degree $d$ of the matrix
polynomial. 

Another interesting observation is that the number of
outer iterations rapidly decreases and stabilizes for values of $q$
greater than $14$. A similar behaviour has the number of inner
iterations which tend to stabilize for $q\ge 14$ as well. This
explains why the graph of the CPU time has an almost linear growth for
$q\ge 14$. In fact, with the number of inner and outer iterations being almost constant, the time spent for the inner iterations is proportional to $q$ and so is the time spent for the outer iterations. 
It is also interesting to observe that the overall number of inner iterations, reported in the graphs at the third column of Figure \ref{fig:1}, tends to stabilize on the value of the overall number of iterations required by the $U$-based method (line marked by red stars). This explains the higher efficiency of the new iteration with respect to the $U$-based method, since the cost of one step of the $U$-based method is proportional to the degree $d$ of the matrix polynomial while the cost of performing an inner iteration is proportional to the degree $q+1$ of the matrix equation which must be solved at each step.
It is important to point out that, while CR is not self-correcting, the methods based on fixed point iterations are self-correcting. In fact, as shown in Table \ref{tab:res}, unlike CR, the new  iteration allows to obtain approximation to $G$ with a smaller residual error.

In Tables \ref{tab:speedup} and \ref{tab:speedupcr} we report the speed--up factor of the CPU time obtained with the different values of $q>1$ with respect to the time needed by the $U$-based iteration and by the CR algorithm, respectively.
It is interesting to observe that the optimal speed up for each problem ranges from 8.5 to 48.4 if compared to the $U$-based method. This value increases as the drift gets close to 0. The speed-up with respect to CR takes large values only for problems which are far from being null recurrent. This happens since CR, unlike functional iteration, is not much depending on the drift of the stochastic process and has an almost constant CPU time. However, it must be said that CR cannot provide the highest accuracy in the approximation as we have already pointed out.

The acceleration in the CPU time can be further increased if we implement the algorithm in a recursive fashion where, instead of the $U$-based iteration to solve the equation of degree $q+1$, we use the same approach by embedding the mass of the matrix polynomial of degree $q+1$ into the leading coefficient of a matrix polynomial of lower degree.

\begin{table}\scriptsize
\begin{center}
\begin{tabular}{l|llllll} 
$\mu$ &\tt -1e-1&\tt -5e-2&\tt -1e-2&\tt -5e-3&\tt -1e-3&\tt -5e-4 \\ \hline  
CR &\tt 1.7e-14 &\tt  1.7e-14 &\tt  1.7e-14 &\tt  1.7e-14 &\tt  1.8e-14 &\tt  1.8e-14 \\ 
New &\tt 8.5e-16 &\tt  9.8e-16 &\tt  9.8e-16 &\tt  9.9e-16 &\tt  9.8e-16 &\tt  9.9e-16 \\ 
\end{tabular}
\end{center}\caption{\footnotesize
Residual errors for different values of the drift in the approximations provided by CR and by the new iteration with error bound $10^{-15}$.}\label{tab:res}
\end{table}

\begin{table} \scriptsize 
\scriptsize
\begin{center}
\begin{tabular}{c|ccccccccc} 
$\mu\backslash q$& 2 &3   & 4   & 5   & 6   & 7   & 8   & 9   & 10\\ \hline 

\tt -1e-1&3.2 & 4.6 & 5.9 & 7.0 & 7.4 & 8.1 & 8.6 & 8.5 & 8.5
\\ 
\tt-5e-2& 3.3 & 5.2 & 7.8 & 9.6 & 11.3 & 12.4 & 13.0 & 13.5 & 13.4\\ 
\tt-1e-2& 4.4 & 7.7 &12.0 &16.9 &21.4 &25.6 &29.3 &32.2 &33.3\\ 
\tt-5e-3& 4.5 & 8.1 &12.9 &18.3 &24.1 &29.3 &33.6 &37.8 &40.0\\ 
\tt-1e-3&4.6 & 8.5 &13.7 &20.2 &26.3 &31.3 &35.3 &38.4 &41.9\\ 
\tt-5e-4&4.7 & 8.6 &14.1 &20.5 &26.6 &31.3 &34.6 &37.6 &39.8\\ 
\end{tabular}\end{center}\caption{\footnotesize
Speed--up of the CPU time of the iteration obtained by embedding the tail in the coefficient of degree $q+1$, with respect to the $U$-based iteration, obtained for some values of the drift $\mu$.}\label{tab:speedup} 
\end{table}

\begin{table} \scriptsize
\scriptsize
\begin{center}
\begin{tabular}{c|ccccccccc} 
$\mu\backslash q$& 2 &3   & 4   & 5   & 6   & 7   & 8   & 9   & 10\\ \hline 

\tt -1e-1&10.9 & 15.1 & 18.0 & 21.4 & 25.2 & 26.8 & 28.5 & 28.3 & 28.3\\ 
\tt-5e-2& 6.8 & 10.2 & 13.7 & 17.1 & 20.3 & 22.0 & 23.0 & 24.2 & 22.2\\ 
\tt-1e-2& 1.8 & 3.2 &5.0 &7.0 &8.7 &10.0 &11.8 &12.6 &14.2\\ 
\tt-5e-3& 1.1 & 1.9 &3.0 &4.4 &5.8 &7.1 &8.4 &9.2 &10.0\\ 
\tt-1e-3& 0.26 & 0.47 &0.78 &1.1 &1.5 &1.8 &2.1 &2.3 &2.5\\ 
\tt-5e-4& 0.14 & 0.27 &0.44 &0.64 &0.85 &1.0 &1.1 &1.2 &1.3\\ 
\end{tabular}\end{center}\caption{\footnotesize
Speed--up of the CPU time of the iteration obtained by embedding the tail in the coefficient of degree $q+1$, with respect to Cyclic Reduction, obtained for some values of the drift $\mu$.}\label{tab:speedupcr} 
\end{table}

In Figure~\ref{fig:x} we report the residual errors per step, for the methods obtained by embedding the tail into the coefficient of degree $q+1$ for a few values of $q$.
This graph extends to higher values of $q$ the graph in Figure~\ref{fig:1}. We may observe that, the larger $q$, the steeper is the slope of the curve. Moreover, the convergence turns out to be linear with a factor which is smaller for larger $q$.

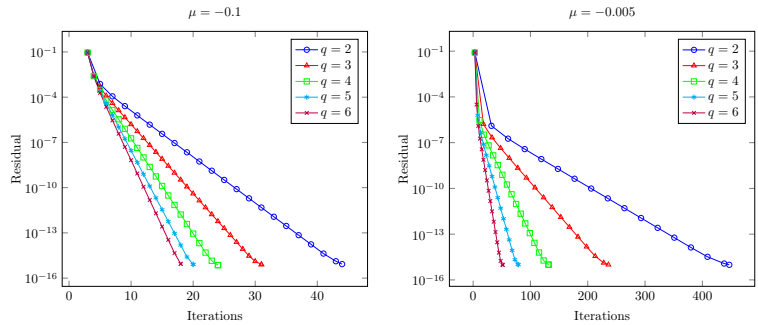
\begin{figure} 
        \begin{center}
        \begin{tikzpicture}[scale=0.50]
        \begin{semilogyaxis}[
                legend pos = north east, width = .8\linewidth, 
                xlabel = {Iterations}, 
                ylabel = {Residual}, title = {$\mu = -0.1$}]
            \addplot[mark=o, color=blue] table {R3_1.dat}; 
            \addplot[mark=triangle, color=red] table {R4_1.dat}; 
            \addplot[mark=square, color=green] table {R5_1.dat}; 
            \addplot[mark=asterisk, color=cyan] table {R6_1.dat}; 
            \addplot[mark=x, color=purple] table {R7_1.dat};
            \legend{$q=2$, $q=3$, $q=4$, $q=5$, $q=6$}
        \end{semilogyaxis}
        \end{tikzpicture}   
~
        \begin{tikzpicture}[scale=0.50]
        \begin{semilogyaxis}[
                legend pos = north east, width = .8\linewidth, 
                xlabel = {Iterations}, 
                ylabel = {Residual}, title = {$\mu = -0.005$}]
            \addplot[mark=o, color=blue] table {R3_4.dat}; 
            \addplot[mark=triangle, color=red] table {R4_4.dat}; 
            \addplot[mark=square, color=green] table {R5_4.dat}; 
            \addplot[mark=asterisk, color=cyan] table {R6_4.dat}; 
            \addplot[mark=x, color=purple] table {R7_4.dat};  
            \legend{$q=2$, $q=3$, $q=4$, $q=5$, $q=6$}
        \end{semilogyaxis}
        \end{tikzpicture}  
\end{center}\caption{\footnotesize
Residual errors per step, for the methods obtained by embedding the tail into the coefficient of degree $q+1$. At each step, an equation of degree $q+1$ is solved. The test problem is the same as the one of Figure \ref{fig:1}.  }\label{fig:x}
\end{figure}

\begin{figure} 
        \begin{center}

                \begin{tikzpicture}[scale=0.4]
        \begin{axis}[
                xlabel = {Degree $q+1$}, 
                ylabel = {CPU time}]
            \addplot[mark=o, color=blue] table {T_php_2_9_0.85_0.dat}; 
            \addplot[mark=diamond, color=red] table {Tu_php_2_9_0.85_0.dat}; 
            \addplot[mark=square, color=green] table {Tcr_php_2_9_0.85_0.dat}; 
        \end{axis}
        \end{tikzpicture}   
~
        \begin{tikzpicture}[scale=0.4]
        \begin{axis}[
                xlabel = {Degree $q+1$}, 
                ylabel = {Outer Iterations}]
            \addplot[mark=o, color=blue] table {Oit_php_2_9_0.85_0.dat};   
        \end{axis}
        \end{tikzpicture} 
~
        \begin{tikzpicture}[scale=0.4]
        \begin{axis}[
                xlabel = {Degree $q+1$}, 
                ylabel = {Inner Iterations}]
            \addplot[mark=o, color=blue] table {Iit_php_2_9_0.85_0.dat}; 
            \addplot[mark=*, color=red] table {itu_php_2_9_0.85_0.dat};  
        \end{axis}
                \end{tikzpicture} 
        \\
                \begin{tikzpicture}[scale=0.4]
        \begin{axis}[
                xlabel = {Degree $q+1$}, 
                ylabel = {CPU time}]
            \addplot[mark=o, color=blue] table {T_php_2_9_0.85_I.dat}; 
            \addplot[mark=diamond, color=red] table {Tu_php_2_9_0.85_I.dat}; 
            \addplot[mark=square, color=green] table {Tcr_php_2_9_0.85_I.dat}; 
        \end{axis}
        \end{tikzpicture}   
~
        \begin{tikzpicture}[scale=0.4]
        \begin{axis}[
                xlabel = {Degree $q+1$}, 
                ylabel = {Outer Iterations}]
            \addplot[mark=o, color=blue] table {Oit_php_2_9_0.85_I.dat};   
        \end{axis}
        \end{tikzpicture} 
~
        \begin{tikzpicture}[scale=0.4]
        \begin{axis}[
                xlabel = {Degree $q+1$}, 
                ylabel = {Inner Iterations}]
            \addplot[mark=o, color=blue] table {Iit_php_2_9_0.85_I.dat}; 
            \addplot[mark=*, color=red] table {itu_php_2_9_0.85_I.dat};  
        \end{axis}
        \end{tikzpicture} 
\end{center}\caption{\footnotesize
CPU time, in seconds, number of outer and inner iterations required to compute $G$ within a residual error less than {\tt 1.e-15} for the PH/PH/1 problem. The two lines denoted by a green square and a red diamond,  mark the CPU time needed by PWCR and by the $U$-based iteration, respectively. 
The line denoted by a red star, in the third column, denotes the number of iterations needed by the $U$-based method.
The graphs in the first row concern the case where $X_0=0$, the graphs in the second row the case where $X_0=I$.
} \label{fig:php}
\end{figure}
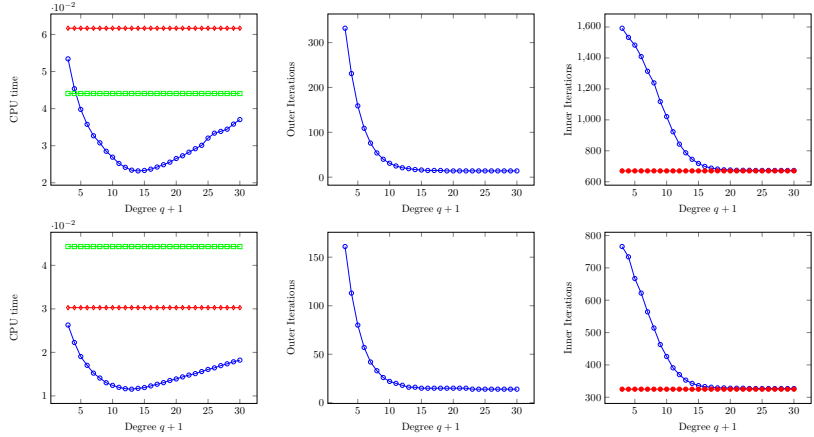

\subsubsection{PH/PH/1 queues}
In this case, we compare the performances of the fixed point iterations, when the starting approximation is the null matrix or a stochastic matrix.


The $U$-based iteration and the new iterations are applied  with $X_0=0$ and $X_0=I$. 
In the first row of Figure \ref{fig:php} we report the case $X_0=0$, in the second row the case $X_0=I$.
We may observe that the same behaviour reported in Figure \ref{fig:h} is maintained. For $X_0=I$ the speed-up, in terms of the CPU time, of the new iteration with respect to the $U$-based method and to CR is 2.6 and 3.8, respectively. For the new iteration, the speed-up obtained for $X_0=I$ with respect to $X_0=0$ is 2.0.

In Tables \ref{tab:phph} and \ref{tab:phphin}, for each value of the degree $q+1$, we have reported the number of outer and inner iterations, respectively, obtained with the two different initial approximations. In the first column, we report the number of iterations 
required by the $U$-based method. In the last column, in boldface, we report the minimum number of outer and inner iterations, respectively, together with the degree $q+1$ of the associated embedding, where the minimum is taken for $q+1$ in the range $[3,d-1]$, being $d$ the degree of the matrix polynomial. Also in this test, the number of inner iterations gets closer to the number of iterations required by the $U$-based method. 

\begin{table}\scriptsize
\begin{center}
\begin{tabular}{cc| ccccccc|c}\\ 
& & \multicolumn{7}{c|}{$q+1$}&\\ 
 &$U$-based &3&4&5&6&7&8&9& \\ \hline
$X_0=0$&670&231 &159 &109 &76 &54 &40 &31 &\bf 14/19\\
$X_0=I$&325&113&80&57&42&33&26&22&\bf 14/23\\
\end{tabular}
\end{center}
\caption{\footnotesize
For different choices of $X_0$, 
number of $U$-based iteration (first column) and
number of outer iterations of the method obtained by embedding the mass in the term of degree $q+1$. In the last column, in bold, the minimum number of outer iterations together with the degree $q+1$ of the corresponding embedding. }\label{tab:phph}
\end{table}

\begin{table}\scriptsize
\begin{center}
\begin{tabular}{cc| cccccc|c}\\ 
& & \multicolumn{6}{c|}{$q+1$}&\\ 
 &$U$-based &3&4&5&6&7&8& \\ \hline
$X_0=0$&670&1592 &1533 &1483 &1409 &1314 &1239 
  &\bf 673/25\\
$X_0=I$&325&766 &734 &667 &622 &564 &514 
 &\bf 327/23\\
\end{tabular}
\end{center}
\caption{\footnotesize
For different choices of $X_0$, 
number of $U$-based iteration (first column) and
number of inner iterations of the method obtained by embedding the mass in the term of degree $q+1$. In the last column, in bold, the minimum number of inner iterations together with the degree $q+1$ of the corresponding embedding.}\label{tab:phphin}
\end{table}

\bibliographystyle{abbrv}

\clearpage

\end{document}